\documentclass[a4paper]{article}

\usepackage[T1]{fontenc}
\usepackage[a4paper, total={6in, 8.7in}]{geometry}
\usepackage{amssymb,amsmath}
\usepackage{dsfont,color}
\usepackage{algorithm}
\usepackage{algpseudocode}
\usepackage{amsthm}

\usepackage{amsmath,amssymb,amsthm,amsfonts}
\usepackage{nicefrac}
\usepackage{graphicx,kantlipsum}
\usepackage{mathtools}
\usepackage{array}

\begin{document}

\newtheorem{thm}{Theorem}[section]
\newtheorem{asm}{Assumption}[section]
\newtheorem{prop}[thm]{Proposition}
\newtheorem{lemma}[thm]{Lemma}
\newtheorem{cor}[thm]{Corollary}
\newtheorem{dfn}{Definition}[section]
\newtheorem{exm}{Example}[section]
\newtheorem{rmk}{Remark}[section]
\allowdisplaybreaks

\title{Finite sample properties of parametric MMD estimation:
\\
robustness to misspecification and dependence
}
\author{Badr-Eddine Ch\'erief-Abdellatif $^{(1)}$ \& Pierre Alquier $^{(2)}$
\\
\small{(1) CNRS and (2) ESSEC Business School}}
\date{\today}

\maketitle

\begin{abstract}
Many works in statistics aim at designing a universal estimation procedure, that is, an estimator that would converge to the best approximation of the (unknown) data generating distribution in a model, without any assumption on this distribution. This question is of major interest, in particular because the universality property leads to the robustness of the estimator. In this paper, we tackle the problem of universal estimation using a minimum distance estimator presented in~\cite{Briol2019} based on the Maximum Mean Discrepancy. We show that the estimator is robust to both dependence and to the presence of outliers in the dataset. Finally, we provide a theoretical study of the stochastic gradient descent algorithm used to compute the estimator, and we support our findings with numerical simulations.

\textcolor{red}{The proof of Proposition~\ref{prop:beta} in the published version contains a mistake. The mistake is fixed here (and the bound is actually improved by a factor $2$).}
\end{abstract}

\section{Introduction}
  
One of the main challenges in statistics is the design of a \textit{universal} estimation procedure. Given data, a universal procedure is an algorithm that provides an estimator of the generating distribution which is simultaneously statistically consistent when the true distribution belongs to the model, and robust otherwise. Typically, a universal estimator is consistent for any model, with minimax-optimal or fast rates of convergence and is robust to small departures from the model assumptions \cite{BickelRobust1976} such as sparse instead of dense effects or non-Gaussian errors in high dimensional linear regression. Unfortunately, most statistical procedures are based upon strong assumptions on the model or on the corresponding parameter set, and very famous estimation methods such as maximum likelihood estimation (MLE), method of moments or Bayesian posterior inference may fail even on simple problems when such assumptions do not hold. For instance, even though MLE is consistent and asymptotically normal with optimal rates of convergence in parametric estimation under suitable regularity assumptions \cite{LeCamMLE70,VanderVaartMLE98} and in nonparametric estimation under entropy conditions, this method behaves poorly in case of misspecification when the true generating distribution of the data does not belong to the chosen model.

Let us investigate a simple example presented in \cite{Birge2006ModelSelectionTesting} that illustrates the non-universal characteristic of MLE. We observe a collection of $n$ independent and identically distributed (i.i.d) random variables $X_1,...,X_n$ that are distributed according to some mixture distribution $P^0_n = (1-2n^{-1})\mathcal{U}([0,1/10]) + 2n^{-1} \mathcal{U}([1/10,9/10])$ where $\mathcal{U}([a,b])$ is the uniform distribution between $a$ and $b$. We consider the parametric model of independent uniform distributions $\mathcal{U}([0,\theta])$, $0 \leq \theta < 1$, and we choose the squared Hellinger distance $h^2(\cdot,\cdot)$ as the risk measure. Here the maximum likelihood is the maximum of the observations $X_{(n)}:=\max(X_1,...,X_n)$, and $\mathcal{U}([0,1/10])$ is a good approximation of the generating distribution $P^0_n$ as $h^2(P^0_n,\mathcal{U}([0,1/10])) < 5/4n$ for $n\geq4$. Hence, one would expect that $\mathbb{E}[h^2(P^0_n,\mathcal{U}([0,X_{(n)}]))]$ goes to $0$ as $n\rightarrow +\infty$, which is actually not the case. We do not even have consistency: $\mathbb{E}[h^2(P^0_n,\mathcal{U}([0,X_{(n)}]))]>0.38$. Hence, the MLE is not robust to this small deviation from the parametric assumption. 
The same happens in Bayesian statistics: the regular posterior distribution is not always robust to model misspecification. Indeed, authors of \cite{Barronetal1999,grunwaldmisspecifiation} show pathologic cases where the posterior does not concentrate to the true distribution.

Universal estimation is all the more important since it provides a generic approach to tackle the more and more popular problem of robustness to outliers under the i.i.d assumption, although definitions and goals involved in robust statistics are quite different from the universal estimation perspective. H\"uber introduced a framework that models situations where a small fraction $\varepsilon$ of data is contaminated, and he assumes that the true generated distribution can be written $(1-\varepsilon)P_{\theta_0}+\varepsilon Q$ where $Q$ is the contaminating distribution and $\varepsilon$ is the proportion of corrupted observations \cite{H\"uberRobustness}. The goal when using this approach is to estimate the true parameter $\theta_0$ given a misspecified model $\{P_{\theta}/\theta\in\Theta\}$ with $\theta_0 \in \Theta$. A procedure is then said to be robust in this case if it leads to a good estimation of the true parameter $\theta_0$. More generally, when a procedure is able to provide a good estimate of the generating distribution of i.i.d data when a small proportion of them is corrupted, whatever the values of these outliers, then such an estimator is considered as robust.

Interestingly enough, none of the aforementioned works questioned the independence assumption on the observation. We believe that a universal estimation procedure should still produce sensible estimations under small deviations from this assumption.

\subsection{Related work}

Several authors attempted to design a general universal estimation method. Sture Holm \cite{BickelRobust1976} suggested that Minimum Distance Estimators (MDE) were the most natural procedures being robust to misspecification. Motivated by \cite{WolfowitzMDE1957,Parr80}, MDE consists in minimizing some probability distance $d$ between the empirical distribution and a distribution in the model. The MDE $\hat{\theta}_n$ is defined by:
$$
d(\hat{P}_n,P_{\hat{\theta}_n}) = \inf_{\theta \in \Theta} d(\hat{P}_n,P_{\theta})
$$
where $\hat{P}_n=n^{-1}\sum_{i=1}^n\delta_{\{X_i\}}$ is the empirical measure and $\Theta$ the parameter set associated to the model. If the minimum does not exist, then one can consider an $\varepsilon$-approximate solution. In fact, this minimum distance estimator is used in many usual procedures. Indeed, the generalized method of moments \cite{Hansen1982} is actually defined as minimizing the weighted Euclidean distance between moments of $\hat{P}_n$ and $P_{\theta}$ while the MLE minimizes the KL divergence, at least for discrete measures. When the distance $d$ is wisely chosen, e.g. when it is bounded, then MDE can be robust and consistent.

A popular metric is the Total Variation (TV) distance \cite{Yatracos1985,DevroyeL1}. \cite{Yatracos1985} built an estimator that is uniformly consistent in TV distance and is robust to misspecification under the i.i.d assumption, but without any assumption on the true distribution of the data. The rate of convergence depends on the Kolmogorov entropy of the model. A few decades later, Devroye and Lugosi studied in details the skeleton estimate, a variant of the estimator of \cite{Yatracos1985} that is based on the TV-distance restricted to the so-called Yatracos sets, see \cite{DevroyeL1}. Unfortunately, the skeleton estimate and the original Yatracos estimate are not computationally tractable.

In \cite{BaraudBirge2016RhoEstimators} and \cite{BaraudBirgeSart2017RhoEstimators}, Baraud, Birg\'e and Sart introduced $\rho$-estimation, a universal method that retains some appealing properties of the MLE such as efficiency under some regularity assumptions, while being robust to deviations, measured by the Hellinger distance. This $\rho$-estimation procedure is inspired from T-estimation \cite{Birge2006ModelSelectionTesting}, itself inspired from earlier works of Le Cam \cite{LeCam73,LeCam75} and Birg\'e \cite{Birge83}, and goes beyond the classical compactness assumption used in T-estimation. In compact models, $\rho$-estimators can be seen as variants of T-estimators also based on robust tests, but they can be extended to noncompact models such as linear regression with fixed or random design with various error distributions. As T-estimators, they enjoy robustness properties, but involve other metric dimensions which lead to optimal rates of convergence with respect to the Hellinger distance even in cases where T-estimators can not be defined. Moreover, when the sample size is large enough, $\rho$-estimation recovers the usual MLE in density estimation when the model is parametric, well-specified and regular enough. Hence, $\rho$-estimation can be seen as a robust version of the MLE. Unfortunately, this strategy is also intractable.


More recently, \cite{Briol2019} showed that using the Maximum Mean Discrepancy (MMD) \cite{Gretton2012} to build a minimum distance estimator leads to both robust estimation in the i.i.d case, without any assumption on the model $\{P_\theta,\theta\in\Theta\}$. Moreover, this estimator is tractable as soon as the model is generative, that is, when one can sample efficiently from any $P_\theta$. MMD, a metric based on embeddings of probability measures into a reproducing kernel Hilbert space, has been applied successfully in a wide range of problems such as kernel Bayesian inference \cite{Song2011}, approximate Bayesian computation \cite{Park2016}, two-sample \cite{Gretton2012} and goodness-of-fit testing \cite{Jit17}, and MMD GANs \cite{Roy2015,li15} and autoencoders~\cite{Zhao2017}, to name a few prominent examples. Such minimum MMD-based estimators are proven to be consistent, asymptotically normal and robust to model misspecification. The trade-off between the statistical efficiency and the robustness is made through the choice of the kernel. The authors investigated the geometry induced by the MMD on a finite-dimensional parameter space and introduced a (natural) gradient descent algorithm for efficient computation of the estimator. This algorithm is inspired from  the stochastic gradient descent (SGD) used in the context of MMD GANs where the usual discriminator is replaced with a two-sample test based on MMD \cite{Roy2015}. These results were extended in the Bayesian framework by~\cite{BECA2019}.

Finally, a whole branch of probability and statistics study limit theorems (LLN, CLT) under the assumptions that the data is not exactly independent, but that in some sense, the dependence between the observations is not strong. Since the seminal work of~\cite{mixing0}, many mixing conditions, that is, restrictions on the dependence between observations, were defined. These conditions lead to limit theorems useful to analyze the asymptotic behavior of estimators computed on time series~\cite{mixing}. Nevertheless, checking mixing assumptions is difficult in practice and many classes of processes that are of interest in statistics such as elementary Markov chains are sometimes not mixing. More recently, \cite{DependenceDoukhan1999} proposed a new weak dependence condition for time series that is built on covariance-based coefficients which are much easier to compute than mixing ones, and that is more general than mixing as it stands for most relevant classes of processes. We believe that it is important to study robust estimators in this setting, in order to check that they are also robust from small deviations to the independence assumption.

\subsection{Contributions}

In this paper, we further investigate universality properties of minimum distance estimation based on MMD distance \cite{Briol2019}. Inspired by the related literature, our contributions in this paper are the following:
\begin{itemize}
\item We go beyond the classical i.i.d framework. Indeed, we prove that the the estimator is robust to dependence between observations. To do so, we introduce a new dependence coefficient expressed as a covariance in some reproducing kernel Hilbert space, and which is very simple to use in practice.
\item We show that our oracle inequalities imply robust estimation under the i.i.d assumption in the H\"uber contamination model and in the case of adversarial contamination.
\item We propose a theoretical analysis of the SGD algorithm used to compute this estimator in \cite{Briol2019} and \cite{Roy2015} for some finite dimensional models. Thanks to this algorithm, we provide numerical simulations to illustrate our theoretical results.
\end{itemize}


The first result of this paper is a generalization bound in the non-i.i.d setting. It states that under a very general dependence assumption, the generalization error with respect to the MMD distance decreases in $n^{-1/2}$ as $n\rightarrow +\infty$. This result extends the inequalities in \cite{Briol2019} that are only available in the i.i.d framework, and is obtained using dependence concepts for stochastic processes. We introduce in this paper a new dependence coefficient in the wake of \cite{DependenceDoukhan1999} which can be expressed as a covariance in some reproducing kernel Hilbert space associated with MMD. This coefficient can be easily computed in many situations and which may be related to usual mixing coefficients such as the popular $\beta$-mixing one. We show that a weak assumption on this new dependence coefficient can relax the i.i.d assumption of \cite{Briol2019} and can lead to valid generalization bounds even in the dependent setting.

Regarding robustness, we prove that our generalization bounds for the MMD estimator implies that this estimator is robust to the presence of outliers. Note that this includes H\"uber's type contamination, and adversarial contamination as well. In particular, we compare the rate of convergence of the MMD estimator to the minimax estimators in the example of the estimation of the mean of a Gaussian.

Regarding computational issues, we provide a Stochastic Gradient Descent (SGD) algorithm as in \cite{Briol2019,Roy2015} involving a U-statistic approximation of the expectation in the formula of the MMD distance. We theoretically analyze this algorithm in parametric estimation using a convex parameter set. We also perform numerical simulations that illustrate the efficiency of our method, especially by testing the behavior of the algorithm in the presence of outliers.

The rest of the paper is organized as follows. Section \ref{sec:notations} defines the MMD-based minimum distance estimator and our new dependence coefficient based on the kernel mean embedding. Section \ref{sec:main-res} provides nonasymptotic bounds in the dependent and misspecified framework, with their implications in terms of robust parametric estimation. Section \ref{sec:examples} illustrates the efficiency of our method in several different frameworks. We finally present an SGD algorithm with theoretical convergence guarantees in Section \ref{sec:algo} and we perform numerical simulations in Section \ref{sec:simu}. The proofs of the theorems of Section~\ref{sec:main-res} are provided in Section \ref{sec:proofs}. The supplementary material is dedicated to the remaining proofs.

 \section{Background and definitions}
\label{sec:notations}
 
In this section, we first introduce some notations and present the statistical setting of the paper in Section 2.1. Then, we remind in Section 2.2 some theory on reproducing kernel Hilbert spaces (RKHS) and we define both the maximum mean discrepancy (MMD) and our minimum distance estimator based on the MMD. Finally, we introduce in Section 2.3 a new dependence coefficient expressed as a covariance in a RKHS.

\subsection{Statistical setting}

We shall consider a dependent setting throughout the paper. We observe in a measurable space $\big( \mathbb{X},\mathcal{X} \big)$ a collection of $n$ random variables $X_1$,...,$X_n$ generated from a stationary process. This implies that the $X_i$'s are identically distributed, and we will let $P^0$ denote their marginal distribution. Note that this include as an example the case where the $X_i$'s are i.i.d with generating distribution $P^0$. We introduce a statistical model $\{ P_{\theta}/ \theta \in \Theta \}$ indexed by a parameter space $\Theta$.

\subsection{Maximum Mean Discrepancy}

We consider a positive definite kernel function $k$, i.e a symmetric function $k : \mathbb{X} \times  \mathbb{X} \rightarrow \mathbb{R}$ such that for any integer $n\geq 1$, for any $x_1,...,x_n \in \mathbb{X}$ and for any $c_1,...,c_n \in \mathbb{R}$:
$$
\sum_{i=1}^n \sum_{j=1}^n c_i c_j k(x_i,x_j) \geq 0.
$$
We then consider the reproducing kernel Hilbert space (RKHS) $({\mathcal{H}_{k}},\langle\cdot,\cdot\rangle_{\mathcal{H}_{k}})$ associated with the kernel $k$ which satisfies the reproducing property $f(x)=\langle f,  k(x,\cdot)\rangle_{\mathcal{H}_{k}}$ for any function $f \in {\mathcal{H}_{k}}$ and any $x \in \mathbb{X}$. From now on, we assume that the kernel is bounded by some positive constant, that will be assumed to be $1$ without loss of generality. Thas is, for any $x,y \in \mathbb{X}$, $|k(x,y)|\leq1$. 


Now we introduce the notion of \textit{kernel mean embedding}, a Hilbert space embedding of a probability measure that can be viewed as a generalization of the original feature map used in support vector machines and other kernel methods. Given a probability measure $P$, we define the mean embedding $\mu_P \in {\mathcal{H}_{k}}$ as:
$$
\mu_P(\cdot) := \mathbb{E}_{X\sim P}[k(X,\cdot)] \in {\mathcal{H}_{k}} .
$$
All the applications and the theoretical properties of those embeddings have been well studied \cite{Fuku2017}. In particular, the mean embedding $\mu_P$ satisfies the relationship $\mathbb{E}_{X\sim P}[f(X)] = \langle f,  \mu_P\rangle_{\mathcal{H}_{k}}$ for any function $f \in {\mathcal{H}_{k}}$, and induces a semi-metric \footnote{ This means that $P \rightarrow \| \mu_P \|_{\mathcal{H}_{k}}$ satisfies the requirements of a norm besides $ \| \mu_P - \mu_Q \|_{\mathcal{H}_{k}} = 0 $ only for $\mu_P=\mu_Q$. } on measures called maximum mean discrepancy (MMD), defined for two measures $P$ and $Q$ as follows:
$$
\mathbb{D}_k(P,Q) = \| \mu_P - \mu_Q \|_{\mathcal{H}_{k}}
$$
or equivalently
$$
\mathbb{D}_k^2(P,Q) = \mathbb{E}_{X,X' \sim P}[k(X,X')] - 2 \mathbb{E}_{X\sim P,Y\sim Q}[k(X,Y)] + \mathbb{E}_{Y,Y'\sim Q}[k(Y,Y')] .
$$
A kernel $k$ is said to be characteristic if $P\mapsto \mu_P$ is injective. This ensures that $\mathcal{D}_k$ is a metric, and not only a semi-metric. Subsection 3.3.1 of the thorough survey \cite{Fuku2017} provides a wide range of conditions ensuring that $k$ is characteristic. They also provide many examples of characteristic kernels, see their Table 3.1. Among others, when $\mathbb{X} \subset \mathbb{R}^d$ equiped with the Euclidean norm $\|\cdot\|$, the Gaussian kernel $k(x,y) = \exp(-\|x-y\|^2/\gamma^2)$ and the Laplace kernel $k(x,y) = \exp(-\|x-y\|/\gamma)$, are known to be characteristic. We actually mostly use these two kernels in our applications. From now on, we will assume that $k$ is characteristic.

Note that there are many applications of the kernel mean embedding and MMD in statistics such as two-sample testing \cite{Gretton2012}, change-point detection \cite{Arlot2012}, detection \cite{MONK19}, we also refer the reader to~\cite{Vert2019} for a thorough introduction to the applications of kernels and MMD to computationnal biology .

Here, we will focus on estimation of parameters based on MMD. This principle was used to train generative networks \cite{Roy2015,li15}, it's only recently that it was studied as a general principle for estimation \cite{Briol2019}. Following these papers we define the MMD estimator $\hat{\theta}_n$ such that:
$$
\mathbb{D}_k(P_{\hat{\theta}_n},\hat{P}_n) = \inf_{\theta \in \Theta} \mathbb{D}_k(P_{\theta},\hat{P}_n)
$$
where $\hat{P}_n=(1/n)\sum_{i=1}^n \delta_{X_i}$ is the empirical measure, i.e.:
$$
\hat{\theta}_n = \underset{\theta \in \Theta}{\arg\min} \, \bigg\{
\mathbb{E}_{X,X' \sim P_\theta}[k(X,X')] - \frac{2}{n} \sum_{i=1}^n \mathbb{E}_{X\sim P_\theta}[k(X,X_i)] \bigg\}.
$$
It could be that there is no minimizer, see the discussion in Theorem 1 page 9 in~\cite{Briol2019}. In this case, we can use an approximate minimizer. More precisely, for any $\varepsilon>0$ we can always find a $\hat{\theta}_{n,\varepsilon}$ such that:
$$
\mathbb{D}_k(P_{\hat{\theta}_{n,\varepsilon}},\hat{P}_n) \leq \inf_{\theta \in \Theta} \mathbb{D}_k(P_{\theta},\hat{P}_n) + \varepsilon.
$$
In what follows, we will consider the case where the minimizer exists (that is, $\varepsilon=0$) but when this is not the case, everything can be easily extended by considering $\hat{\theta}_{n,1/n}$.

\subsection{Covariances in RKHS}

In this subsection, we introduce and discuss a new dependence coefficient based on the kernel mean embedding. This coefficient allows to go beyond the i.i.d case in the study of the MMD estimator of \cite{Briol2019}, and to show that it is actually robust to dependence.

\begin{dfn}
 \label{dfn.varrho}
 We define, for any $t\in\mathbb{N}$,
 $$ \varrho_t = \left| \mathbb{E}\left< k(X_t,\cdot)-\mu_{P^0},k(X_0,\cdot)-\mu_{P^0} \right>_{\mathcal{H}_k} \right| . $$
\end{dfn}
In the i.i.d case, note that $\varrho_t = 0$ for any $t\geq 1$. In general, the following assumption will ensure the consistency of our estimator:
\begin{asm}
 \label{asm:our:mixing}
 There is a $\Sigma < + \infty$ such that, for any $n$, $\sum_{t=1}^n \varrho_t \leq \Sigma$.
\end{asm}

Our mean embedding dependence coefficient may be seen as a covariance expressed in the RKHS $\mathcal{H}_{k}$. We shall see throughout the paper that the kernel mean embedding coefficient $\varrho_t$ can be easily computed in many situations, and that it is closely related to widely used mixing coefficients. In particular, we will show in Section 4.2 that our coefficient $\varrho_t$ is upper-bounded by the popular $\beta$-mixing coefficient. For the reader who would not be familiar with $\beta$-mixing, we also show that any real-valued auto-regressive process $X_t = a X_{t-1} + \varepsilon_t$ satisfies Assumption~\ref{asm:our:mixing} as long as $|a|<1$, the $\varepsilon_t$ are i.i.d and $\mathbb{E}(|\varepsilon_0|)<\infty$. Also, we show that some special cases of such auto-regressive processes are not $\beta$-mixing, which proves that Assumption~\ref{asm:our:mixing} is more general than $\beta$-mixing: an explicit example is given in Subsection~\ref{subsec:ar}. Hence, Assumption \ref{asm:our:mixing} may be referred to as a weak dependence condition in the wake of the concept of weak dependence introduced in \cite{DependenceDoukhan1999}. We will show in the next section that under Assumption \ref{asm:our:mixing}, we can obtain a nonasymptotic generalization bound of the same order than in the i.i.d case.

 \section{Nonasymptotic bounds in the dependent, misspecified case}
\label{sec:main-res}
 
In this section, we provide nonasymptotic generalization bounds in MMD distance for the minimum MMD estimator. In particular, we show in Subsection~\ref{subsec:mmd} that under a weak dependence assumption, it is robust to both dependence and misspecification, and is consistent at the same $n^{-1/2}$ rate than in the i.i.d case. In particular, we give explicit bounds in the H\"uber contamination model and in a more general adversarial setting in Subsection~\ref{subsec:robust}.

\subsection{Estimation with respect to the MMD distance}
\label{subsec:mmd}

First, we begin with a theorem that gives an upper bound on the generalization error, i.e the expectation of $\mathbb{D}_k(P_{\hat{\theta}_n},P^0)$. The rate of convergence of this error is of order $n^{-1/2}$ independently of the dimension of the parameter space $\Theta$. In fact, note that there is actually no assumption at all on the model $\{P_{\theta},\theta\in\Theta\}$ in this theorem.
\begin{thm}
\label{theorem:mmd:1}
 We have: $$ \mathbb{E} \left[  \mathbb{D}_k\left(P_{\hat{\theta}_n},P^0 \right) \right]  \leq \inf_{\theta\in\Theta}  \mathbb{D}_k\left(P_{\theta},P^0 \right) + 2 \sqrt{ \frac{1+2\sum_{t=1}^{n}\varrho_t}{n}} . $$
\end{thm}
As a consequence, under Assumption~\ref{asm:our:mixing}:
 $$ \mathbb{E} \left[  \mathbb{D}_k\left(P_{\hat{\theta}_n},P^0 \right) \right]  \leq \inf_{\theta\in\Theta}  \mathbb{D}_k\left(P_{\theta},P^0 \right) + 2 \sqrt{ \frac{1+2\Sigma }{n}} . $$
We remind that the proofs of the results in this section are deferred to Section~\ref{sec:proofs}. It is also possible to provide a result that holds with large probability as in \cite{Briol2019,Roy2015}. Naturally, this requires stronger assumptions, and the conditions on the dependence become more intricate in this case. Here, we use a condition introduced in \cite{McDoRIO,McDoRIO2} for generic metric spaces that we adapt to the kernel embedding and to stationarity:
\begin{asm}
\label{asm:gamma:mixing}
Assume that there is a family $(\gamma_{\ell})_\ell$ of non-negative numbers such that, for any integer $n$, for any $\ell\in\{1,\dots,n-1\} $ and any function $g:\mathcal{H}_k^\ell \rightarrow \mathbb{R} $ such that
$$ |g(a_1,\dots,a_\ell) - g(b_1,\dots,b_\ell)| \leq \sum_{i=1}^\ell \|a_i - b_i\|_{\mathcal{H}_k} , $$
we have: $ |\mathbb{E}[g(\mu_{\delta_{X_{\ell+1}}},\dots,\mu_{\delta_{X_{n}}})|X_{1},\dots,X_{\ell}]-\mathbb{E}[g(\mu_{\delta_{X_{\ell+1}}},\dots,\mu_{\delta_{X_{n}}})]| \leq \gamma_{1} + \dots + \gamma_{n+\ell-1} $, almost surely. Assume that $\Gamma:= \sum_{\ell\geq 1} \gamma_{\ell} < \infty $.
\end{asm}
This assumption is more technical than Assumption~\ref{asm:our:mixing}. The idea is quite similar: the coefficient $\gamma_s$ is a measure of the dependence between $X_t$ and $X_{t+s}$, and the assumption will be satisfied if $X_t$ and $X_{t+s}$ are ``almost independent'' when $s$ is large -- but the sense given to ``almost independent'' is not exactly the same as in Assumption~\ref{asm:our:mixing}. For example, we show in Subsection~\ref{subsec:ar} that auto-regressive processes $X_{t+1}=a X_t + \varepsilon_{t+1}$ with $|a|<1$ and i.i.d $\varepsilon_t$ satisfy this assumption under the additional condition that the $\varepsilon_t$ are almost surely bounded. Again, note that in the case of independence, we can take all the $\gamma_{i}=0$ and hence $\Gamma=0$ in addition to $\Sigma=0$. We can now state our result in probability:
\begin{thm}
\label{theorem:mmd:briol:improved}
Assume that Assumptions~\ref{asm:our:mixing} and~\ref{asm:gamma:mixing} are satisfied. Then, for any $\delta\in(0,1)$,
 $$ \mathbb{P} \left[ \mathbb{D}_k\left( P_{\hat{\theta}_n},P^0 \right) \leq  \inf_{\theta\in\Theta} \mathbb{D}_k\left( P_{\theta},P^0 \right)
 + 2 \frac{\sqrt{1+2\Sigma} + (1+\Gamma)\sqrt{2\log\left(\frac{1}{\delta}\right)} }{\sqrt{n}} \right] \geq 1-\delta.$$
\end{thm}
Assumption \ref{asm:gamma:mixing} is fundamental to obtain a result in probability. Indeed, the rate of convergence in Theorem  \ref{theorem:mmd:briol:improved} is characterized by some concentration inequality upper bounding the MMD distance between the empirical and the true distribution as done in \cite{Briol2019}. Nevertheless, the proof of this inequality in \cite{Briol2019} is based on a Hoeffding-type inequality known as McDiarmid's inequality \cite{McDo} that is only valid for independent variables (that is, all the $\gamma_i=0$), which makes this inequality not applicable in our dependent setting. Hence we use a version of McDiarmid's inequality for time series obtained by Rio \cite{McDoRIO,McDoRIO2} which is available under the assumption that $\sum_{\ell\geq 1} \gamma_{\ell} < \infty$ (Assumption \ref{asm:gamma:mixing}).

\begin{rmk}[The i.i.d case]
 Note that when the $X_i$'s are i.i.d, Assumptions~\ref{asm:our:mixing} and~\ref{asm:gamma:mixing} are always satisfied with $\Sigma=\Gamma=0$ and thus Theorem~\ref{theorem:mmd:1} gives simply
 $$ \mathbb{E} \left[  \mathbb{D}_k\left(P_{\hat{\theta}_n},P^0 \right) \right]  \leq \inf_{\theta\in\Theta}  \mathbb{D}_k\left(P_{\theta},P^0 \right) +\frac{2}{\sqrt{n}}  $$
 while Theorem~\ref{theorem:mmd:briol:improved} gives
 $$ \mathbb{P} \left[ \mathbb{D}_k\left( P_{\hat{\theta}_n},P^0 \right) \leq  \inf_{\theta\in\Theta} \mathbb{D}_k\left( P_{\theta},P^0 \right)
 + 2 \frac{1 + \sqrt{2\log\left(\frac{1}{\delta}\right)} }{\sqrt{n}} \right] \geq 1-\delta.$$ 
\end{rmk}

\begin{rmk}[Connection between the MMD distance and the $L^2$ norm]
\label{rmk:L2}
In Section~\ref{sec:examples}, we study the connection between the convergence of $\hat{P}_{\hat{\theta}_n}$ in terms of MMD distance, and the convergence of $\hat{\theta}_n$, is some classical models. However, it is also worth mentioning a connection between the MMD distance and the quadratic distance on densities. Indeed, assume $\mathbb{X} = \mathbb{R}^d$ and that $P$ and $Q$ have density $p$ and $q$ respectively with respect to the Lebesgue measure. Using the Gaussian kernel $k_\gamma(x,y) = \exp(-\|x-y\|^2/\gamma^2)$, we expect that, when $\gamma \rightarrow 0$, under suitable assumptions,
$$ \mathbb{E}_{X\sim P, Y\sim Q}[k(X,Y)] \sim \pi^{\frac{d}{2}} \gamma^d \int p(x) q(x) {\rm d}x  $$
and so that
\begin{equation}
\label{equa:lien:l2}
\mathbb{D}_{k_\gamma}(P,Q) \sim \pi^{\frac{d}{4}} \gamma^{\frac{d}{2}} \|p-q\|_{L^2}.
\end{equation}
Corollary 4 page 1527 of~\cite{L2} provides a formal statement of this claim. Thus, the convergence in the MMD distance has connections with the convergence of the densities (when they exist) in $L^2$.

Note that~\cite{DevroyeL1,BaraudBirge2016RhoEstimators} argue that the $L^2$-norm is not suitable for universal estimation: indeed, in some models, $P_\theta$ does not have a density with respect to the Lebesgue measure. But~\eqref{equa:lien:l2} allows the interpretation of the MMD distance (with the Gaussian kernel) as an approximation of the $L^2$ distance, that is however well defined for {\it any} model $(P_\theta)$.
\end{rmk}

\subsection{Robust parametric estimation}
\label{subsec:robust}

\subsubsection{Contamination models}

As explained in the introduction, when all observations but a small proportion of them are sampled independently from a generating distribution $P_{\theta_0}$ ($\theta_0 \in \Theta$), robust parametric estimation consists in finding estimators being both rate optimal and resistant to outliers. Two among the most popular frameworks for studying robust estimation are the so-called H\"uber's contamination model and the adversarial contamination model. 

H\"uber's contamination model is as follows. We observe a collection of random variables $X_1,...,X_n$. We consider a contamination rate $\varepsilon\in(0,1/2)$, latent i.i.d random variables $Z_1,...,Z_n \sim \text{Ber}(\varepsilon)$ and some noise distribution $Q$, such that the distribution of $X_i$ given $Z_i=0$ is $P_{\theta_0}$, and that the distribution of $X_i$ given $Z_i=1$ is $Q$. Hence, the observations $X_i$'s are independent and sampled from the mixture $P^0=(1-\varepsilon)P_{\theta_0}+\varepsilon Q$.

The adversarial model is more general. Contrary to H\"uber's contamination where outliers were all sampled from the contaminating distribution, we do not make any particular assumption on the outliers here. Hence, we shall adopt slightly different notations. We assume that $X_1,\dots,X_n$ are identically distributed from $P_{\theta_0} $ for some $\theta_0\in\Theta$. However, the statistician only observes $\tilde{X_1},\dots,\tilde{X_n}$ where $\tilde{X}_i$ can be any arbitrary value for $i\in \mathcal{O}$, where $\mathcal{O}$ is an arbitrary set subject to the constraint $|\mathcal{O}| \leq \varepsilon n$, and $\tilde{X}_i=X_i$ for $i\notin \mathcal{O}$. The estimators are built based on these observations $\tilde{X_1},\dots,\tilde{X_n}$.

\subsubsection{Literature}

One hot research trend in robust statistics is focused on the search of both statistically optimal and computationally tractable procedures for the Gaussian mean estimation problem $\{P_\theta=\mathcal{N}(\theta,I_d)/\theta \in \mathbb{R}^d\}$ in the presence of outliers under the i.i.d assumption, which remains a major challenge. Usual robust estimators such as the coordinatewise median and the geometric median are known to be suboptimal in this case, and there is a need to look at more complex estimators such as Tukey's median that achieves the minimax optimal rate of convergence $\max(\frac{d}{n},\varepsilon^2)$ with respect to the squared Euclidean distance, where $d$ is the dimension, $n$ is the sample size and $\varepsilon$ is the proportion of corrupted data. Unfortunately, computation of Tukey's median is not tractable and even approximate algorithms lead to an $\mathcal{O}(n^d)$ complexity \cite{ChanTukey2004,AmentaTukey2000}. This has led to the rise of the recent studies in robust statistics which address how to build robust and optimal statistical procedures, in the wake of the works of \cite{Tukey1975} and \cite{H\"uberRobustness}, but that are also computationally efficient.

This research area started with two seminal works presenting two procedures for the normal mean estimation problem: the \textit{iterative filtering} \cite{DiakonikolasRobust2016} and the \textit{dimension halving} \cite{LaiRaoVempala2016}. These algorithms are based upon the idea of using higher moments in order to obtain a good robust moment estimation, and are minimax optimal up to a poly-logarithmic factor in polynomial time. This idea was then used in several other problems in robust statistics, for instance in sparse functionals estimation \cite{DuRobustFunctionals}, clustering \cite{KothariRobustClustering}, mixtures of spherical Gaussians learning \cite{DiakonikolasExtension1}, and robust linear regression \cite{DiakonikolasExtension2}. In H\"uber's contamination model, \cite{ColDal17a} achieves the minimax rate without any extra factor in the $\varepsilon = \mathcal{O}(\min(d^{-1/2},n^{-1/4}))$ regime with an improved overall complexity. Meanwhile, \cite{ChaoGaoRobustGAN2019} offers a different perspective on robust estimation and connects the robust normal mean estimation problem with Generative Adversarial Networks (GANs) \cite{GoodfellowGAN2014,BiauGAN2018}, what enables computing robust estimators using efficient tools developed for training GANs. Hence, the authors compute depth-like estimators that retain the same appealing robustness properties than Tukey's median and that can be trained using stochastic gradient descent (SGD) algorithms that were originally designed for GANs.

Another popular approach for the more general problem of mean estimation under the i.i.d assumption in the presence of outliers is the study of finite-sample sub-Gaussian deviation bounds. Indeed, designing estimators achieving sub-Gaussian performance under minimal assumptions ensures robustness to outliers that are inevitably present when the generating distribution is heavy-tailed. In the univariate case, some estimators present a sub-Gaussian behavior for all distributions under first and second order moments. A simple but powerful strategy, the Median-of-Means (MOM), dates back to \cite{Nemi1983,Jer86,Alon1999}. This method consists in randomly splitting the data into several equal-size blocks, then computing the empirical mean within each block, and finally taking the median of them. Most MOM-based procedures lead to estimators that are simultaneously statistically optimal \cite{Lugosi2016,MOM1,Lecue2018,MONK19,chinot2019} and computationally efficient \cite{Hopkins2019,chera2019,depersin2019}. Moreover, this approach can be easily extended to the multivariate case \cite{Minsker2015,Hsu2016}. An important advantage is that the MOM estimator has good performance even for distributions with infinite variance. An elegant alternative to the MOM strategy is due to Catoni, whose estimator is based on PAC-Bayesian truncation in order to mitigate heavy tails \cite{Catoni2012}. It has the same performance guarantees than the MOM method but with sharper and near-optimal constants. In \cite{CatoniGiulini2017}, Catoni and Giulini proposed a very simple and trivial-to-compute multidimensional extension of Catoni's M-estimator defined as an empirical average of the data, with the observations with large norm shrunk towards zero, and that still satisfies a sub-Gaussian concentration using PAC-Bayes inequalities. The influence function of Catoni and Giulini has been widely used since then, see \cite{Ilaria2017,Ilaria2018,Holland2019a,Holland2019b,Haddouche2020}. We refer the reader to the beautiful review of \cite{Lugosi2019mean} for more details on those mean estimation procedures.

\subsubsection{Robust MMD estimation}

In this section, we show the properties of our MMD-based estimator in robust parametric estimation with outliers, both in H\"uber's contamination model and in the adversarial case. Our bounds are obtained by working directly in the RKHS rather than in the parameter space. the consequence of these results in terms of the Euclidean distance in the parameter space will be explored in Section~\ref{sec:examples}. 

First we consider H\"uber's contamination model \cite{H\"uberRobustness}. The objective is to estimate $P_{\theta_0}$ by observing contaminated random variables $X_1$, ..., $X_n$ with actual distribution is $P^0 = (1-\alpha) P_{\theta_0} + \alpha Q $ for some $Q$, and some $0\leq \alpha \leq \varepsilon$. We state the key following lemma:

\begin{lemma}
\label{lemma:huber}
We have, for any $\theta\in\Theta$, $ | \mathbb{D}_k(P_{\theta},P^0) - \mathbb{D}_k (P_{\theta},P_{\theta_0})| \leq 2 \varepsilon$.
\end{lemma}

Applying Lemma~\ref{lemma:huber} to the left-hand side, and to the right-hand side, of Theorem~\ref{theorem:mmd:1}, we have the following result.

\begin{cor}
Assume that $X_1,\dots,X_n$ are identically distributed from $P^0 = (1-\alpha) P_{\theta_0} + \alpha Q $ for some $\theta_0\in\Theta$, some $Q$, with $0\leq \alpha \leq \varepsilon$. Then:
 $$ \mathbb{E} \left[  \mathbb{D}_k\left(P_{\hat{\theta}_n},P_{\theta_0} \right) \right]  \leq 4\varepsilon  + 2 \sqrt{ \frac{1+2\sum_{t=1}^{n}\varrho_t}{n}} . $$
If moreover we assume that Assumptions~\ref{asm:our:mixing} and~\ref{asm:gamma:mixing} are satisfied, then for any $\delta\in(0,1)$,
 $$ \mathbb{P} \left[ \mathbb{D}_k\left( P_{\hat{\theta}_n},P_{\theta_0}\right) \leq  2 \left( 2\varepsilon 
 +  \frac{\sqrt{1+2\Sigma} + (1+\Gamma)\sqrt{2\log\left(\frac{1}{\delta}\right)} }{\sqrt{n}} \right) \right] \geq 1-\delta.$$
\end{cor}

We obtain a rate $\max(1/\sqrt{n},\varepsilon)$ in MMD distance (note once again that the convergence rate with respect to more standard distances is studied in Section~\ref{sec:examples}). When $\varepsilon \lesssim 1/\sqrt{n}$, then we recover the rate of convergence without contamination, and when $1/\sqrt{n} \lesssim \varepsilon$, then the rate is dominated by the contamination ratio $\varepsilon$. Hence, the maximum number of outliers which can be tolerated without breaking down the rate is $n\varepsilon \asymp \sqrt{n}$.

This result can also be extended to the adversarial contamination setting, where no assumption is made on the outliers.

\begin{prop}
 \label{prop:adversarial}
 Assume that $X_1,\dots,X_n$ are identically distributed from from $P^0 = P_{\theta_0} $ for some $\theta_0\in\Theta$. However, the statistician only observes $\tilde{X_1},\dots,\tilde{X_n}$ where $\tilde{X}_i$ can be any arbitrary value for $i\in \mathcal{O}$, $\mathcal{O}$ is any arbitrary set subject to the constraint $|\mathcal{O}| \leq \varepsilon n$, and $\tilde{X}_i=X_i$ for $i\notin \mathcal{O}$ and builds the estimator $\tilde{\theta}_n$ based on these observations:
 $$ \mathbb{D}_k\left(P_{\tilde{\theta}_n},\frac{1}{n}\sum_{i=1}^n \delta_{\tilde{X}_i}\right) = \inf_{\theta \in \Theta} \mathbb{D}_k\left(P_{\theta},\frac{1}{n}\sum_{i=1}^n \delta_{\tilde{X}_i}\right). $$
 Then:
 $$  \mathbb{D}_k\left(P_{\tilde{\theta}_n},P_{\theta_0} \right)  \leq 
 4\varepsilon  + 2 \mathbb{D}_k\left(P_{\hat{\theta}_n},P_{\theta_0} \right) .  $$
Thus
 $$
 \mathbb{E} \left[  \mathbb{D}_k\left(P_{\tilde{\theta}_n},P_{\theta_0} \right) \right]  \leq 4\varepsilon  + 4 \sqrt{ \frac{1+2\sum_{t=1}^{n}\varrho_t}{n}}
 $$
 and, under Assumptions~\ref{asm:our:mixing} and~\ref{asm:gamma:mixing}, for any $\delta\in(0,1)$,
 $$ \mathbb{P} \left[ \mathbb{D}_k\left( P_{\hat{\theta}_n},P_{\theta_0}\right) \leq  4 \left( \varepsilon 
 +  \frac{\sqrt{1+2\Sigma} + (1+\Gamma)\sqrt{2\log\left(\frac{1}{\delta}\right)} }{\sqrt{n}} \right) \right] \geq 1-\delta.$$
\end{prop}

One can see that the rate of convergence we obtain without making any assumption on the outliers is exactly the same than in H\"uber's contamination model. The only difference is that the constant in the right hand side of the inequality is tighter in H\"uber's contamination model.

 \section{Examples}
\label{sec:examples}
 
\subsection{Independent observations}

In the previous section we studied the convergence of $P_{\hat{\theta}_n}$ with the MMD distance. In this subsection, we show what are the consequences of these results in terms of the convergence of $\hat{\theta}$ in some classical models. For the sake of simplicity, we focus on i.i.d observations. That is, $\varrho_t = 0$ for any $t\geq 1$. Moreover, we will only use the Gaussian kernel $k_\gamma(x,y) = \exp(-\|x-y\|^2/\gamma^2)$.

\subsubsection{Estimation of the mean in a Gaussian model}

Here, $\mathbb{X}=\mathbb{R}^d$ and we are interested in the estimation of the mean in a Gaussian model. For the sake of simplicity, we assume that the variance is known.

\begin{prop} \label{prop:ex:gauss}
 Assume that $P_\theta = \mathcal{N}(\theta,\sigma^2 I_d)$ for $\theta\in\Theta=\mathbb{R}^d$. Moreover, assume that we are in an adversarial contamination model where a proportion at most $\varepsilon$ of the observations is contaminated. Then, with probability $1-\delta$,
\begin{equation} \label{ex:gauss:3}
\|\theta-\tilde{\theta}_n\|^2
\leq - (4\sigma^2 + \gamma^2)
 \log\left\{ 1-8 {\rm e}^{\frac{2\sigma^2 d}{\gamma^2}} \left( \varepsilon 
 +  \frac{1 + \sqrt{2\log\left(\frac{1}{\delta}\right)} }{\sqrt{n}} \right)^2 \right\}.
\end{equation}
In particular, the choice $\gamma = \sigma \sqrt{2d}$ leads to
$$
\|\tilde{\theta}_n - \theta_0 \|^2 \leq -2\sigma^2(d+2) \log\left[1-8 {\rm e} \left( \varepsilon 
 +  \frac{1 + \sqrt{2\log\left(\frac{1}{\delta}\right)} }{\sqrt{n}} \right)^2\right].
$$

\end{prop}
The complete proof can be found in the supplementary material. Note that when $\varepsilon$ is small and $n$ is large,
\begin{multline*}
\|\tilde{\theta}_n-\theta_0\|^2 \leq -2\sigma^2(d+2) \log\left[1-16 {\rm e} \left( \varepsilon^2 
 +  \frac{\left(1 + \sqrt{2\log\left(\frac{1}{\delta}\right)}\right)^2 }{n} \right)\right]
\\
\sim
32 {\rm e}\sigma^2(d+2) \left( \varepsilon^2 
 +  \frac{\left(1 + \sqrt{2\log\left(\frac{1}{\delta}\right)}\right)^2 }{n} \right).
\end{multline*}

We can see that our MMD estimator achieves a rate of convergence $d\varepsilon^2 + d/n$ which is the same than for several median-based estimators such as the geometric median or the coordinatewise median (see Proposition 2.1 in \cite{ChenGao2018}). We have a quadratic dependence in $\varepsilon$, contrary to many robust estimators such as Median-of-Means which dependence in $\varepsilon$ is linear. Hence, as soon as the dimension is no larger than the square root of the sample size $d\leq\sqrt{n}$, the MMD method tolerates a larger number of outliers without affecting the usual rate of convergence (i.e. the rate with no contamination).

Unfortunately, it seems that our method performs poorly compared to such estimators in large dimension. Indeed, according to Theorems 2.1 and 2.2 in \cite{ChenGao2018}, the minimax optimal rate with respect $d$, $\varepsilon$ and $n$ is $\varepsilon^2 + d/n$. Furthermore, numerical experiments and the investigation conducted for the population limit case when one has access to infinitely many samples in \cite{MMDGANRobust} (that has been published since the first version of this paper) suggest that the MMD estimator can not match the minimax rate of convergence. Nevertheless, this non-optimality in the minimax sense does not necessarily imply inaccurate mean estimation in general, and MMD can still lead to efficient estimation in most contamination scenarios.

To understand why the MMD estimator can not match the minimax rate of convergence in high dimension, and why this is not necessarily a problem, we need to analyze the landscape of the optimization program. 

Let us first investigate the population limit case, where we do not work with the MMD distance to the empirical distribution $\hat{P}_n$ but to the true distribution $(1-\varepsilon)\mathcal{N}(\theta_0,\sigma^2 I_d)+\varepsilon Q$, as if we had access to infinitely many samples, and with a point-mass delta Dirac contamination $Q=\delta_{\{\theta_c\}}$. The optimization program is, for any value of $\gamma$, 
\begin{multline*}
\min_{\theta\in\mathbb{R}^d} \mathbb{D}_{k_\gamma}\left(P_{\theta},(1-\varepsilon)\mathcal{N}(\theta_0,\sigma^2 I_d)+\varepsilon \delta_{\{\theta_c\}}\right)
\\
=
\max_{\theta\in\mathbb{R}^d} \bigg\{
(1-\varepsilon) \exp\left(-\frac{\|\theta -\theta_0\|^2}{\gamma^2+4\sigma^2}\right) + \varepsilon \left(\frac{\gamma^2+4\sigma^2}{\gamma^2+2\sigma^2}\right)^{d/2} \exp\left(-\frac{\|\theta-\theta_c\|^2}{\gamma^2+2\sigma^2}\right) \bigg\} .
\end{multline*}

Even though the objective function is nonconvex in $\theta$, it is easy to see that the solution belongs to the line between $\theta_0$ and $\theta_c$. More precisely, if $\theta_0$ and $\theta_c$ are far from each other, then the solution is simply $\theta_0$. At the opposite, if $\theta_0$ and $\theta_c$ are closed, then the solution will be very close to $\theta_0$. In the situation in between where $\|\theta_0-\theta_C\|^2\approx d$, then it is proven in \cite{MMDGANRobust} that the solution is at least $\varepsilon\sqrt{d}$ far from the true parameter $\theta_0$, which explains the term $d\varepsilon^2$ in the rate of convergence of the MMD estimator. Hence, we understand that the worst-case rate of the MMD estimator does not correspond to cases where $\theta_c$ is far from $\theta_0$ but to cases where the distance is quite large in high dimensions only (of order $\sqrt{d}$).

The previous reasoning can be easily generalized to the MMD estimator with a finite sample. In this situation with $Q=\delta_{\{\theta_c\}}$, the optimization program can be written, denoting by $\mathcal{O}$ the set of outliers, 
$$
\max_{\theta\in\mathbb{R}^d} \bigg\{
\sum_{i\notin\mathcal{O}} \exp\left(-\frac{\|\theta -X_i\|^2}{\gamma^2+2\sigma^2}\right) + |\mathcal{O}| \exp\left(-\frac{\|\theta-\theta_c\|^2}{\gamma^2+2\sigma^2}\right) \bigg\} , 
$$
and the solution belongs to the convex hull of the set of points composed of the (random) inliers in the random variables $X_1,...,X_n$ and of the contamination point $\theta_c$. A remarkable point in high dimensional probability is that samples from a multivariate standard Gaussian distribution are concentrated on the sphere of radius $\sqrt{d}$ centered at $\theta_0$, which means that the typical distance $\|X_i-\theta_0\|$ of a datapoint $X_i$ from the mean $\theta_0$ is roughly $\sqrt{d}$. Then, if the contamination is such that $\|\theta_0-\theta_c\|
^2\approx d$, the outliers lie at a distance $\sqrt{d}$ from $\theta_0$ without being detected, and thus look harmless but shift the mean by approximately $\sqrt{d}\varepsilon$, see Figure \ref{shiftmean}.
\begin{figure}[h]
\caption{Illustration of the behaviour of the MMD estimator in the high-dimensional Gaussian mean estimation problem. The true parameter $\theta_0$ and datapoints sampled from the true distribution $\mathcal{N}(\theta_0,I_d)$ are colored in blue. Outliers  and the MMD estimator $\hat{\theta}_n$ are colored in red. We can see that outliers lying at a distance $\sqrt{d}$ are not detected and shift the mean by $\varepsilon\sqrt{d}$.
}
\label{shiftmean}
\includegraphics[width=7cm]{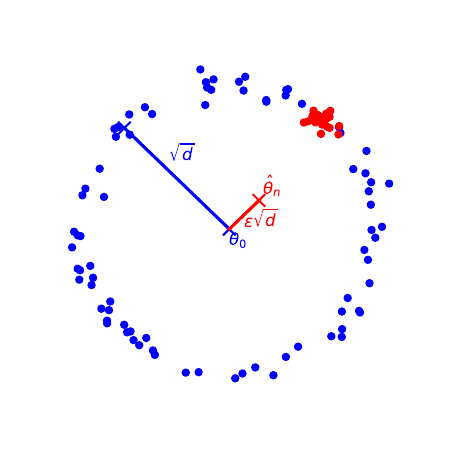}
\centering
\end{figure}

Hence, perhaps counter-intuitively  at first sight, the worst contamination does not come from a value of $\theta_c$ that is very far away from $\theta_0$ (in which case the estimation will simply be the mean of the inliers), but that is only $\sqrt{d}$ away from $\theta_0$, and hence there is mainly one "worst-case contamination" that explains the non-optimality in the minimax sense. Figure 1a of \cite{MMDGANRobust} even seems to show that the error of the MMD estimator when $\gamma$ is of order $\sqrt{d}$ increases with $\|\theta_0-\theta_c\|$ until achieving $\sqrt{d}$, and then decreases. The same applies to a Gaussian contamination with a small variance.

\subsubsection{Cauchy model}

Here, $\mathbb{X}=\mathbb{R}$ and $P_\theta=\mathcal{C}(\theta,1)$ where $\mathcal{C}(\theta,s)$ has density $1/[\pi s (1+(x-\theta)^2 / s^2)]$.

\begin{prop} \label{prop:ex:cauchy}
 Assume that $P_\theta = \mathcal{C}(\theta,1)$ for $\theta\in\Theta=\mathbb{R}$. Moreover, assume that we are in an adversarial contamination model where a proportion at most $\varepsilon$ of the observations is contaminated. Then, taking $\gamma=2$ leads to, for any $\delta>0$,
$$
(\tilde{\theta}_n - \theta_0)^2 \leq
4\left( 1 - \frac{1}{1-96 \pi \left( \varepsilon^2 + \frac{2 + 4\log(1/\delta) }{n} \right) } \right).
$$
\end{prop}
Note that
$$ (\tilde{\theta}_n - \theta_0)^2 \leq
4\left( 1 - \frac{1}{1-128 \pi \left( \varepsilon^2 + \frac{2 + 4\log(1/\delta) }{n} \right) } \right) \sim 512 \pi \left( \varepsilon^2 + \frac{2 + 4\log(1/\delta) }{n} \right) .$$

\subsubsection{Estimation with a dictionary}
\label{dictionary}

We consider here estimation of $P^0$ by a linear combination of measures in a dictionary. This framework actually appears in various models:
\begin{itemize}
 \item first, when the dictionary contains probability distributions, this is simply a mixture of known components. In this case, the linear combination is actually a convex combination. This context is for example studied in~\cite{Dal2017}.
 \item assuming that $P^0$ has a density, in nonparametric density estimation, we can use this setting, the dictionary being defined by a basis of $L_2$. This is for example the point of view in~\cite{Alquier2008,BTW1,BTW2}.
\end{itemize}
We will here focus on the first setting, but an extension to the second one is quite straightforward. Let $\{\Phi_1,\dots,\Phi_D\}$ be a family of probability measures over $\mathbb{X}=\mathbb{R}^d$. For $1\leq i\leq D$ we remind that
$$ \mu_{\Phi_i}(\cdot) = \int  k(x,\cdot) \Phi_i({\rm d}x). $$
Define the measure $P_\theta=\mathcal{D}(\theta;\Phi_1,\dots,\Phi_D)=\sum_{i=1}^D \theta_i \Phi_i $, and define the model $\{P_\theta,\theta\in\Theta\}$ with $\Theta \subseteq \mathcal{S}_D=\{\theta\in\mathbb{R}_+^D: \sum_{i=1}^D \theta_i = 1 \}$. The estimator is then
$$
\hat{\theta}_n = \underset{\theta \in \Theta}{\arg\min} \left\| \sum_{\ell=1}^D \theta_\ell \mu_{\Phi_\ell}(\cdot) -\mu_{\hat{P}_n} \right\|^2_{\mathcal{H}_k}.
$$
An application of Theorem~\ref{theorem:mmd:briol:improved} leads to:

\begin{prop}
 Assume that $P_{\theta} = \sum_{i=1}^D \theta_i \Phi_i$ where $\Phi_i$ is a probability distribution. Define the matrix $G_\gamma = (\left<\mu_{\Phi_i},\mu_{\Phi_j}\right>_{\mathcal{H}_{k_\gamma}} )_{1\leq i,j\leq D} $. Letting $\lambda_{\min}(\cdot) $ denote the smallest eigenvalue of a symmetric matrix, we have:
 $$ \mathbb{P} \left[ \mathbb{D}_k\left( P_{\hat{\theta}_n},P^0 \right) \leq  \inf_{\theta\in\Theta} \mathbb{D}_k\left( P_{\theta},P^0 \right)
 + 2 \frac{1 + \sqrt{2\log\left(\frac{1}{\delta}\right)} }{\sqrt{n}} \right] \geq 1-\delta$$
 and, in the well specified case where $P^0 = P_{\theta_0}$,
 $$ \mathbb{P} \left[ \|\hat{\theta}-\theta_0\|^2 \leq  2 \frac{1 + \sqrt{2\log\left(\frac{1}{\delta}\right)} }{\lambda_{\min}(G_\gamma) \sqrt{n}} \right] \geq 1-\delta
 .$$
\end{prop}

\subsection{$\beta$-mixing observations}

We now consider non-independent random variables: as in the general framework presented above, $(X_t)_{t\in \mathbb{Z}}$ is a strictly stationary time series, with stationary distribution $P^0$, and we observe $X_1,\dots,X_n$. We will exhibit some condition on the dependence of the $X_i$'s ensuring that we can still estimate $P^0$ with the MMD method.

There is a very rich literature on limit theorems and exponential inequalities under conditions on various dependence coefficients. Mixing coefficients and their applications are detailed in the monographs~\cite{mixing,mixing2}, weak dependence coefficients in~\cite{weakd}. In this subsection, we show that our coefficient $\varrho_t$ can be upper-bounded by the $\beta$-mixing coefficients. So for any $\beta$-mixing process, the estimation of $P^0$ using MMD remains possible. We also remind some examples of $\beta$-mixing processes.  Note that we will show in the next subsection that Theorem~\ref{theorem:mmd:1} can be successfully applied to non $\beta$-mixing processes.

\subsubsection{$\beta$-mixing and coefficients $\varrho_t$}

We start by a reminder of the definition of the $\beta$-mixing coefficients, from page 4 (Chapter 1) in~\cite{weakd}.
\begin{dfn}
\label{dfn:beta}
Given two $\sigma$-algebras $\mathcal{A}$ and $\mathcal{B}$,
 $$ \beta(\mathcal{A},\mathcal{B}) = \frac{1}{2} \sup_{
 \begin{tiny}
 \begin{array}{c}
 I,J \geq 1
 \\ U_1,\dots,U_I
 \\ V_1,\dots,V_J
 \end{array}
 \end{tiny}
 } \sum_{1\leq i \leq I} \sum_{1\leq j \leq J} |\mathbb{P}(U_i\cap V_j) - \mathbb{P}(U_i)\mathbb{P}(V_j)|  $$
 where $(U_1,\dots,U_I)$ is any partition of $\mathcal{A}$ and $V_1,\dots,V_j$ any partition of $\mathcal{B}$. Put:
 $$\beta_t^{(X)} = \beta(\sigma(X_0,X_{-1},\dots),\sigma(X_t,X_{t+1},\dots)). $$
\end{dfn}
Section 1.5 in~\cite{mixing} provides summability conditions on the $\beta_t^{(X)}$ leading to a law of large numbers and to a central limit theorem. Examples are also discussed.
\begin{exm}
Assume in this example that $(X_t)$ is an homogeneous Markov chain given by its transition kernel $P$ and $X_0\sim \pi$ where $\pi P = \pi$. Assume that there is a $0<c\leq 1$ and a probability measure $Q$ on $\mathbb{R}^d$ such that, for some integer $r\geq 1$ and for any measurable $A$, $P^r(x,A) \geq c Q(A)$. Then it is known, see e.g. Theorem 1 page 88 in~\cite{mixing} that
$$ \beta_t^{(X)} \leq 2(1-c)^{\frac{t}{r}-1}. $$
\end{exm}

We now compare our $\varrho$ coefficients with the $\beta$-mixing coefficients.
\begin{prop}
\label{prop:beta}
Assume that $k(x,y)=F(\|x-y\|)$ were $F(a) = \int_{a}^{\infty} f(b) {\rm d}b $ for some nonnegative continuous function $f$ with $\int_{0}^\infty f(b) {\rm d}b = 1$. Then we have
 $$ \textcolor{red}{\varrho_t \leq  \beta(\sigma(X_0),\sigma(X_t)) \leq  \beta_t.} $$
\end{prop}

Note that $k(x,y) = \exp(-\|x-y\|/\gamma)$ and $k(x,y) = \exp(-\|x-y\|^2/\gamma^2)$ for example trivially work, respectively with $f(b)=\exp(-b/\gamma)/\gamma$ and $f(b)=2 b \exp(-b^2/\gamma^2)/\gamma^2$.

\subsubsection{Application: hidden Markov chains}

Assume here that $(Y_t)_{t\in\mathbb{N}}$ is a Markov chain on $\{1,\dots,d\}$, and that $X_t|(Y_t=i) $ is independent from all the other values $Y_{t'}$ and is drawn in $\mathbb{R}^D$ from a probability measure $\Phi_i$. The $\Phi_i$'s are known and $X_1,\dots,X_n$ are observed but the $(Y_t)_{t\in\mathbb{N}}$ are not observed. Note that this is a dependent extension of the mixture model $\mathcal{D}(\theta;\Phi_1,\dots,\Phi_d)$ discussed above. Indeed, we consider this as a case of misspecification: the statistician uses the mixture model $\mathcal{D}(\theta;\Phi_1,\dots,\Phi_d)$ with $\Theta=\mathcal{S}_d$, being not aware that the data is actually not independent.

Letting $P$ denote the transition matrix of $Y$, we assume that there exists $c>0$ and an integer $r\geq 0$ such that $P^{r}(i,j)\geq c/d$ for any $(i,j)\in\{1,\dots,d\}^2$. Then we have $\beta_t^{(Y)} \leq 2(1-c)^{t/r-1} $. This also implies that there is a unique $\pi$ such that $\pi P = \pi$ and we assume that $Y_0\sim \pi$. Then the distribution $P^0$ of each $X_t$ is given by $P^0(x) = \sum_{i=1}^d \pi_i \Phi_i(x) $.

Also, note that
$$ \varrho_t = \beta(\sigma(X_0),\sigma(X_t)) \leq \beta(\sigma(X_0,Y_0),\sigma(X_t,Y_t)) = \beta(\sigma(Y_0),\sigma(Y_t)) \leq 2(1-c)^{t/r-1} .$$
So, a direct application of Theorem~\ref{theorem:mmd:1} gives:
$$
 \mathbb{E}\left[ \|G_k^{-1/2}(\hat{\theta}-\pi) \|\right] =
\mathbb{E} \left[  \mathbb{D}_k\left(P_{\hat{\theta}_n},P^0 \right) \right]  \leq 2 \sqrt{ \frac{1 + (1-c)^{\frac{1}{r}-1}(3+c)}{n[1-(1-c)^{\frac{1}{r}}]}} .
$$

Note that we can add a second layer in the process: assume that an opponent is allowed to replace a fraction $\varepsilon$ of the $X_t$, as in Proposition~\ref{prop:adversarial}. This result in the observation of $\tilde{X}_t$ such that $\tilde{X}_t= X_t$ for a proportion $(1-\varepsilon)$ of the data, and $\tilde{X}_t$ can be anything for the remaining $\varepsilon$. For example, the opponent can try fo fool the learner, by drawing from the wrong $\Phi_i$. The MMD estimator $\tilde{\theta}$ still satisfies, from Proposition~\ref{prop:adversarial},
$$
\mathbb{E} \left[  \mathbb{D}_k\left(P_{\hat{\theta}_n},P^0 \right) \right]  \leq 4\varepsilon + 4 \sqrt{ \frac{1 + (1-c)^{\frac{1}{r}-1}(3+c)}{n[1-(1-c)^{\frac{1}{r}}]}} .
$$

\subsection{Auto-regressive observations}
\label{subsec:ar}

In this section, we provide examples of auto-regressive processes satisfying Assumption~\ref{asm:our:mixing} and Assumption~\ref{asm:gamma:mixing}, which allows to apply Theorem~\ref{theorem:mmd:briol:improved}. Interestingly, for one of these examples,
$\beta_t = 1/4$ and so $\sum_{t=1}^\infty \beta_t = \infty$, but still $\sum_{t=1}^\infty \varrho_t < \infty$: this means that Assumption~\ref{asm:our:mixing} is more general than $\beta$-mixing.

\subsubsection{Auto-regressive processes}

\begin{prop}
\label{prop:counterexample}
 Assume that $X_t$ takes values in $\mathbb{R}^d$ and that $k(x,y)=F(\|x-y\|)$ where $F$ is an $L$-Lipschitz function. Assume that 
 $$ X_{t+1} = A X_t + \varepsilon_{t+1}  $$
 where the $(\varepsilon_t)$ are i.i.d with $\mathbb{E}\|\varepsilon_0\|< \infty$, and $A$ is a matrix with $\|A\|= \sup_{\|x=1\|} \|A x\| < 1$. Then Assumption~\ref{asm:our:mixing} is satisfied with
 $$ \varrho_t \leq \|A\|^t \frac{2 L \mathbb{E}\|\varepsilon_0\|}{1-\|A\|} \text{ and } \Sigma = \sum_{t=1}^\infty \varrho_t = \frac{2 \|A\| L \mathbb{E}\|\varepsilon_0\|}{(1-\|A\|)^2}. $$
 Moreover, assume that almost surely, $\|\varepsilon_t\|\leq c$. Then Assumption~\ref{asm:gamma:mixing} is satisfied with
 $$ \gamma_i = \frac{2 c \sqrt{L} \|A\|^{\frac{i}{2}}}{1-\|A\|} \text{ and } \Gamma = \sum_{i=1}^{\infty} \gamma_i = \frac{2 c \sqrt{L \|A\|} }{(1-\|A\|)(1-\sqrt{\|A\|})} .$$
\end{prop}

\subsubsection{Examples of non-mixing processes with $\sum_t \varrho_t <\infty $}

First, we remind a classical example of non-mixing process, in the sense that $\sum_{t=1}^\infty \beta_t = \infty$. See for example Section 1.5 page 8 in~\cite{weakd} where it is also proven that it is neither $\alpha$-mixing. The process is real-valued, it is defined by $X_{t+1} = (X_t+ \eta_{t+1})/2$, where the $\eta_t$ are i.i.d $\mathcal{B}e(1/2)$ and $X_0\sim \mathcal{U}([0,1])$. Note that the noise is there $\varepsilon_t=\eta_t/2$. As for any $t$, $X_{t} = f(X_{t+1}) $ where $f$ is the measurable function $f(x) = 2x - \lfloor 2x\rfloor$, it is possible to take $I=J=2$, $V_1=U_1$ and $V_2 = U_2 = U_1^c$ for some $U_1$ with $\mathbb{P}(U_1)=1/2$ in  Definition~\ref{dfn:beta}. This leads to $\beta(\sigma(X_0),\sigma(X_t)) \geq 1/4 $.

However, according to Proposition~\ref{prop:counterexample}, as $\mathbb{E}|\varepsilon_0| = 1/4$, we have
$$ \varrho_t \leq \frac{L}{2^{t}} \text{, } \Sigma = \sum_{t=1}^\infty \varrho_t =  2 L <\infty \text{, and } \gamma_i = 2^{1-\frac{i}{2}}\sqrt{L} \text{, } \Gamma = \sum_{i=1}^\infty \gamma_i = \frac{2 \sqrt{L}}{\sqrt{2}-1} < \infty.$$

Another classical example of non-mixing process is a reversed version of the previous one. We draw $X_0\sim \mathcal{U}([0,1])$ and simply define $X_{t+1}= f(X_t)$ where we still have $f(x)=2x-\lfloor 2x \rfloor$. Note that once $X_0$ is given, the process $(X_t)_{t\geq 0}$ is entirely deterministic, and thus non-mixing. Properties of (generalized versions) of such processes are studied in Section 3.3 page 28 in~\cite{weakd}. It is not difficult to check that $Y_t = X_{-t}$ actually satisfies $Y_{t+1} = (Y_t + \varepsilon_{t})/2$ where the $\varepsilon_t$ are independent $\mathcal{B}e(1/2)$. Thus, a straightforward adaptation of the proof of Proposition~\ref{prop:counterexample} leads to $\varrho_t \leq 2/L^{t}$.

 \section{Stochastic gradient algorithm for MMD estimation}
\label{sec:algo}
 
In this section, we briefly discuss gradient-based algorithms to compute the estimator $\hat{\theta}_n$ when $\Theta\subset\mathbb{R}^d$. In Subsection~\ref{subsec:gradient} we provide an expression of the gradient of the criterion to be minimized. We briefly provide a special case where this gradient can be computed explicitly. However, in general, this is not the case, but we can provide unbiased estimators of this gradient as soon as we are able to sample from $P_\theta$, in this case the model is often referred to as a {\it generative model}. Thus it is possible to use a stochastic gradient algorithm when $\{P_\theta,\theta\in\Theta\}$ is a generative model. We describe this algorithm in Subsection~\ref{subsec:SGA}, and remind its theoretical properties in Subsection~\ref{subsec:thm:algo}.

Note that the idea to use a stochastic gradient algorithm to compute $\hat{\theta}_n$ was first used to train a generative neural network by~\cite{Roy2015}. In~\cite{Briol2019} the authors propose to use a stochastic natural gradient algorithm instead. By providing adaptation to the geometry of the problem, the natural gradient will lead to better results but increase the computational burden when the dimension of the problem is large.

\subsection{Gradient of the MMD distance}
\label{subsec:gradient}

We remind that in this whole section, $\Theta\subset\mathbb{R}^d$. To compute $\hat{\theta}_n$, one must minimize, with respect to $\theta\in\Theta$,
$$ 
\mathbb{D}_k^2(P_\theta,\hat{P}_n)
= \mathbb{E}_{X,X'\sim P_\theta} [k(X,X')]
- \frac{2}{n}\sum_{i=1}^n \mathbb{E}_{X\sim P_\theta}[k(X_i,X)]
+  \frac{1}{n^2}\sum_{1\leq i,j\leq n} k(X_i,X_j)
$$
or, equivalently,
$$
{\rm Crit}(\theta)
= \mathbb{E}_{X,X'\sim P_\theta} [k(X,X')]
- \frac{2}{n}\sum_{i=1}^n \mathbb{E}_{X\sim P_\theta}[k(X_i,X)].
$$
In order to use gradient algorithms or any first order method, a first step is to compute the gradient of this quantity with respect to $\theta$.

\begin{prop}
\label{prop:gradient}
 Assume that each $P_\theta$ has a density $p_{\theta}$ with respect to the Lebesgue measure. Assume that for any $x$, $\theta\mapsto p_{\theta}(x)$ is differentiable with respect to $\theta$ and that there is a nonnegative function $g(x,x')$ such that, for any $\theta\in\Theta$, $|k(x,x')\nabla_{\theta}[p_{\theta}(x) p_{\theta}(x')] | \leq g(x,x') $ and $\iint g(x,x') \mu({\rm d}x )\mu({\rm d}x')< \infty$. Then
 $$
 \nabla_\theta {\rm Crit}(\theta)
 =
 2 \mathbb{E}_{X,X'\sim P_\theta} \left[ \left( k(X,X') - \frac{1}{n}\sum_{i=1}^n k(X_i,X) \right) \nabla_\theta[\log p_\theta(X) ]\right].
 $$
\end{prop}
Note that the gradient of ${\rm Crit}(\theta)$ is given by an expectation with respect to $P_\theta$. So, as soon as it is feasible to sample from $P_\theta$, on can provide unbiased estimates of $\nabla_\theta {\rm Crit}(\theta)$, and thus implement a stochastic gradient algorithm.
\begin{rmk}
It might be that in special cases, we have explicit formulas for the expectations in ${\rm Crit}(\theta)$ and its gradient. For example, assume that we are a translation parameter, that is: $p_\theta(x) = f(x-\theta)$ for some density $f$, and that the kernel $k$ is given by $k(x,x')=K(x-x')$ for some function $K$. Then 
\begin{align*}
 {\rm Crit}(\theta)
& = \iint K(x-x') f(x-\theta) f(x'-\theta)\mu({\rm d}x)\mu({\rm d}x') 
- \frac{2}{n}\sum_{i=1}^n \int K(X_i-x) f(x-\theta) \mu({\rm d}x) \\
& = \iint K(x-x') f(x) f(x')\mu({\rm d}x)\mu({\rm d}x') 
- \frac{2}{n}\sum_{i=1}^n \int K(\theta+x-X_i) f(x) \mu({\rm d}x).
\end{align*}
For example, in the case $P_\theta = \mathcal{U}[\theta-1/2,\theta+1/2]$ we have
\begin{align*}
 {\rm Crit}(\theta)
&  =
\iint_{[-1/2,1/2]^2} K(x-x') {\rm d}x {\rm d}x'
- \frac{2}{n}\sum_{i=1}^n \int_{-1/2}^{1/2} K(\theta+x-X_i){\rm d}x \\
& = \iint_{[-1/2,1/2]^2} K(x-x') {\rm d}x {\rm d}x'
- \frac{2}{n}\sum_{i=1}^n \int_{\theta-1/2-X_i}^{\theta+1/2-X_i} K(u) {\rm d}u
\end{align*}
and thus
$$
\nabla_\theta  {\rm Crit}(\theta) = -\frac{2}{n}\sum_{i=1}^n [K(\theta+1/2-X_i)-K(\theta-1/2-X_i)].
$$
So, in this special case, the estimation of the gradient is unnecessary and we can use a proper gradient algorithm to compute $\hat{\theta}_n$.
\end{rmk}

\subsection{Projected stochastic gradient algorithm for the MMD estimator}
\label{subsec:SGA}

From Proposition~\ref{prop:gradient},
 $$
 \nabla_\theta {\rm Crit}(\theta)
 =
 2 \mathbb{E}_{X,X'\sim P_\theta} \left[ \left( k(X,X') - \frac{1}{n}\sum_{i=1}^n k(X_i,X) \right) \nabla_\theta[\log p_\theta(X) ]\right].
 $$
So, if we can compute $\nabla[\log p_\theta(x) ]$ and if it is feasible to simulate from $P_\theta$, we can easily compute a Monte Carlo estimator of $\nabla_\theta {\rm Crit}(\theta)$ and thus use a stochastic gradient descent (SGD). First, simulate $(Y_1,\dots,Y_M)$ i.i.d from $P_\theta$, then put
$$
\widehat{\nabla_\theta {\rm Crit}} (\theta)
= \frac{2}{M}\sum_{j=1}^M \left( \frac{1}{M-1} \sum_{\ell \neq j} k(Y_j,Y_\ell) - \frac{1}{n} \sum_{i=1}^nk(X_i,Y_j) \right)  \nabla_\theta[\log p_\theta(Y_j) ].
$$
We now provide the details of a projected stochastic gradient algorithm (PSGA). The projection step is necessary if $\Theta\subsetneq \mathbb{R}^d$. Thus, we assume that $\Theta\subseteq\mathbb{R}^d$ is a closed and convex subset and let $\Pi_\Theta$ denote the orthogonal projection on $\Theta$.

\begin{algorithm}
\caption{PSGA for MMD}
\begin{algorithmic}[1]
\State \textbf{Input}: a dataset $(X_1,...,X_n)$, a model $(P_\theta,\theta\in\Theta\subset\mathbb{R}^d)$ a kernel $k$, a sequence of steps $(\eta_t)_{t\geq 1}$, an integer $M$, a stopping time $T$.
\State \textbf{Initialize} $\theta^{(0)}\in\Theta$, $t=0$.
\State \textbf{For} $t=1,\dots,T$
   \State draw $(Y_1,\dots,Y_M)$ i.i.d from $P_{\theta^{(t-1)}}$,
   \State $\theta^{(t)} = \Pi_{\Theta} \left\{ \theta^{(t-1)} - \frac{2 \eta_t }{M}\sum_{j=1}^M \left[ \frac{1}{M-1} \sum_{\ell \neq j} k(Y_j,Y_\ell) - \frac{1}{n} \sum_{i=1}^n  k(X_i,Y_j) \right]  \nabla_{\theta^{(t-1)}}[\log p_{\theta^{(t-1)}}(Y_j) ] \right\} $
\State \textbf{End for}
\end{algorithmic}
\end{algorithm}

\subsection{Theoretical analysis of the algorithm}
\label{subsec:thm:algo}

In its original version, the stochastic gradient algorithm was proposed with a sequence of steps $(\eta)_t$ such that $\eta_t \rightarrow 0$ and $\sum_{t}\eta_t = \infty$. However,~\cite{Nemi2009} proved that the method can be made more robust by taking a constant step size $\eta_t = \eta$ and by averaging the parameters. The following proposition is actually a direct application of the results of~\cite{Nemi2009}.

\begin{prop}
 \label{prop:SGA}
Under the conditions of Proposition~\eqref{prop:gradient} above, and under the assumption that $\Theta$ is closed, convex and bounded with $\mathcal{D} = \sup_{(\theta,\theta')\in\Theta^2} \|\theta-\theta'\| $, define
$$ \hat{\theta}^{(T)}_n = \frac{1}{T}\sum_{t=1}^T \theta^{(t)} $$
where the $\theta^{(t)}$'s are given by Algorithm~1 above. Assume that, for any $\theta\in\Theta$,
$$ \mathbb{E} \left[ \| \widehat{\nabla_\theta {\rm Crit}} (\theta)\|^2\right] \leq M^2. $$
Assume that ${\rm Crit}(\theta)$ is a convex function of $\theta$. Then the choice $\eta=\mathcal{D}/(M \sqrt{T})$ leads to
\begin{equation}
\label{eq:nemi}
\mathbb{E}\left[ {\rm Crit}(\hat{\theta}^{(T)}_n) - {\rm Crit}(\hat{\theta}_n) \right]  \leq \frac{\mathcal{D} M}{\sqrt{T}} ,
\end{equation}
where the expectation $\mathbb{E}$ is taken with respect to drawings of the $Y_i$'s in Algorithm 1. Moreover
$$ \mathbb{E}\left[ \mathbb{D}_{k}\left(P_{\hat{\theta}^{(T)}_n},P^0\right) \right] \leq \inf_{\theta\in\Theta}  \mathbb{D}_{k}(P_{\theta},P^0) + 3\sqrt{\frac{1+2\sum_{t=1}^n \varrho_t }{n}} + \sqrt{\frac{\mathcal{D}M}{\sqrt{T}}} $$
where the expectation is taken with repect to the sample and to the $Y_i$'s, and the choice $T=n^2$ leads to
$$ \mathbb{E}\left[ \mathbb{D}_{k}\left(P_{\hat{\theta}^{(n^2)}_n},P^0\right) \right] \leq \inf_{\theta\in\Theta}  \mathbb{D}_{k}(P_{\theta},P^0) + \frac{\sqrt{\mathcal{D}M} + 3\sqrt{1+2\sum_{t=1}^n \varrho_t }}{n} .$$
\end{prop}

The restrictive assumption in this proposition is the convexity assumption on the criterion. However, it is satisfied in some of the examples of Section~\ref{sec:examples}.

\begin{exm}
Let us come back to the ``estimation with a dictionary'' example of Section~\ref{sec:examples}: $P_\theta$ is given by its density
$$ p_{\theta} = \sum_{\ell=1}^D \theta_\ell \Phi_\ell .$$
Let us assume that $\Theta=\mathcal{S}_D$ and the $\Phi_\ell$'s are probability densities. Then $\Theta$ closed, convex and bounded with $\mathcal{D}=1$. Moreover,
$$
\widehat{\nabla_\theta {\rm Crit}} (\theta)
= \frac{2}{M}\sum_{j=1}^M \left[ \frac{1}{M-1} \sum_{\ell \neq j} k(Y_j,Y_\ell) - \frac{1}{n} \sum_{i=1}^nk(X_i,Y_j) \right]  \nabla_\theta[\log p_\theta(Y_j) ].
$$
and
$$
\nabla_\theta[\log p_\theta(Y_j) ]
=
\left(
\begin{array}{c}
 \frac{\Phi_1(Y_j)}{\sum_{\ell=1}^D \theta_\ell \Phi_\ell(Y_j) }
 \\
 \vdots
 \\
 \frac{\Phi_D(Y_j)}{\sum_{\ell=1}^D \theta_\ell \Phi_\ell(Y_j) }
\end{array}
\right).
$$
Consequently,
\begin{align*}
\left\|
\widehat{\nabla_\theta {\rm Crit}} (\theta)
\right\|^2
& = \sum_{\ell=1}^D \left( \frac{2}{M}\sum_{j=1}^M \left[ \frac{1}{M-1} \sum_{\ell \neq j} k(Y_j,Y_\ell) - \frac{1}{n} \sum_{i=1}^nk(X_i,Y_j) \right]   \frac{\Phi_\ell(Y_j)}{\sum_{\ell=1}^D \theta_\ell \Phi_\ell(Y_j) }  \right)^2
\\
& \leq\sum_{\ell=1}^D \frac{4}{M^2} \sum_{1\leq j,k \leq M}   \frac{\Phi_\ell(Y_j)\Phi_\ell(Y_k)}
{\left(\sum_{\ell=1}^D \theta_\ell \Phi_\ell(Y_j)\right) \left(\sum_{\ell=1}^D \theta_\ell \Phi_\ell(Y_k)\right) }
\\
& =  \frac{4}{M^2} \sum_{1\leq j,k \leq M}   \frac{\sum_{\ell=1}^D \Phi_\ell(Y_j)\Phi_\ell(Y_k)}
{p_{\theta}(Y_j) p_{\theta}(Y_k) },
\end{align*}
and then
\begin{align*}
 \mathbb{E} \left( \left\|
\widehat{\nabla_\theta {\rm Crit}} (\theta)
\right\|^2 \right)
& \leq \iint 4   \frac{\sum_{\ell=1}^D \Phi_\ell(y)\Phi_\ell(y')}
{p_{\theta}(y) p_{\theta}(y') } p_{\theta}(y) p_{\theta}(y') {\rm d}y {\rm d}y'
\\
& = 4 \iint \sum_{\ell=1}^D \Phi_\ell(y)\Phi_\ell(y') {\rm d}y {\rm d}y'
\\
& = 4 D.
\end{align*}
Hence Proposition~\ref{prop:SGA} leads to
$$ \mathbb{E}\left[ {\rm Crit}(\hat{\theta}^{(T)}_n) - {\rm Crit}(\hat{\theta}_n) \right]  \leq \sqrt{\frac{4D}{T}} = 2\sqrt{\frac{D}{T}} .  $$
\end{exm}

\begin{rmk}
In many examples, the MMD criterion is not convex, so we cannot apply Proposition~\ref{prop:SGA}. This includes for example the estimation of the mean of a Gaussian distribution. However, the simulation study below shows that the stochastic gradient algorithm still provides excellent results, even though we cannot prove that it reached a global minimum of the criterion.

In some sense, the situation is similar to the estimation of mixtures with the EM algorithm: we cannot prove that the algorithm will not get trapped in a local minimum, but the algorithm is still extremely valuable in practice. Note that many strategies were proposed in order to avoid local minima for EM: for example, multiple runs of the algorithm, with randomized initializations. This strategy works well, at least in reasonnably small dimension, even though improvements are possible~\cite{EM2012}.
\end{rmk}

 \section{Simulation study}
\label{sec:simu}

In this section, we test our stochastic gradient algorithm in the Gaussian mean estimation and in the Gaussian mixture estimation settings. In all the experiments, we chose a number of Monte-Carlo samples equal to $n$ and a step-size of $\eta_t = 1/\sqrt{t}$, and we used the Gaussian kernel $k(x, y) = e^{-\|x-y\|_2^2/d}$ where $d$ is the dimension. Each experiment is repeated 10 times.



\subsection{Gaussian mean estimation}
First, we estimate the mean of a Gaussian distribution $\mathcal{N}(\mathbf{0},I_d)$ where $I_d$ is the identity matrix of dimension $d$ and where more generally, $\mathbf{a}$ is the vector with all components equal to $a\in\mathbb{R}$. We provide numerical experiments to illustrate and validate our theory on the non-minimax optimality given in Section \ref{sec:examples}. We verify the rates we obtained from a theoretical point of view by exploring via numerical experiments various types of contamination distributions $Q$ and different proportions of outliers. The MMD gradient descent is compared with the maximum likelihood estimator which is here the arithmetic mean, the componentwise median, and the JS-GAN studied in \cite{ChaoGaoRobustGAN2019}, which is known to be robust and even minimax optimal. Note that in \cite{ChaoGaoRobustGAN2019}, JS-GAN outperforms iterative filtering and dimension halving in all experiments, so we don't include these two methods here. The metric considered here is the square root of the mean square error (MSE) over all the 50 repetitions. We focus on the scenario where $d/n<\varepsilon$ with $d=10$, $n=500$ and $\varepsilon=0.2$. We are interested here in the influence of the contamination distribution $Q$ on the MSE. The results are reported on Table \ref{table:only1worstcase}, where the bold character marks the lowest MSE among all methods for each contamination Q. As expected, we can see that there is mainly one "worst-case contamination" (the point-mass distribution in a $\sqrt{d}$-far contamination parameter) for which the MMD estimator performs poorly. At the opposite, in all other situations, the MMD estimator is one of the best methods and is not really affected by the distance of the contamination parameter to $\theta_0$. In particular, MMD estimation is competitive with the minimax-optimal JS-GAN procedure.

\begin{table}[ht]
 \centering 
\begin{tabular}{c c c c c c c c c} 
\hline\hline 
Method & $\mathcal{N}(\mathbf{0.2},I_d)$ & $\mathcal{N}(\mathbf{0.5},I_d)$ & $\mathcal{N}(\mathbf{1},I_d)$ & $\mathcal{N}(\mathbf{5},I_d)$ & $\mathcal{N}(\mathbf{10},I_d)$ & $\mathcal{C}(\mathbf{0.5})$ & $\delta_{\{\mathbf{1}\}}$ & $\delta_{\{\mathbf{10}\}}$ \\ [0.5ex] 
\hline 
Mean & \textbf{0.0379} & 0.0954 & 0.2033 & 1.0166 & 1.9915 & 0.3577 & 0.2057 & 2.0048 \\ 
 & \textbf{(0.0046)} & (0.0039) & (0.0115) & (0.0145) & (0.0153) & (0.6451) & (0.0115) & (0.0156) \\ 
Median & 0.0387 & \textbf{0.0893} & 0.1756 & 0.3106 & 0.3345 & \textbf{0.0769} & 0.3194 & 0.3258 \\
 & (0.0158) & \textbf{(0.0098)} & (0.0058) & (0.0109) & (0.0164) & \textbf{(0.0232)} & (0.0215) & (0.0098) \\ 
JS-GAN & 0.1848 & 0.2036 & 0.2172 & 0.1879 & 0.2204 & 0.2276 & \textbf{0.1969} & 0.1877 \\ 
 & (0.0443) & (0.0346) & (0.0241) & (0.0287) & (0.0423) & (0.0376) & \textbf{(0.0342)} & (0.0324) \\ 
MMD & 0.0654 & 0.1172 & \textbf{0.1730} & \textbf{0.0634} & \textbf{0.0681} & 0.0882 & 0.3622 & \textbf{0.0601} \\ 
 & (0.0132) & (0.0199) & \textbf{(0.0077)} & \textbf{(0.0081)} & \textbf{(0.0157)} & (0.0140) & (0.0212) & \textbf{(0.0157)} \\ 
[1ex] 
\hline 
\end{tabular}
\label{table:only1worstcase} 
\caption{Square root of the MSE for several choices of $Q$ (with standard deviations in parenthesis from the 50 repeated experiments.). Here, $\varepsilon=0.2$, $d=10$ and $n=500$ such that $\sqrt{d/n}<\varepsilon$. Chosen structure of the network for the GAN: 1 hidden layer with 5 hidden units (as suggested in \cite{ChaoGaoRobustGAN2019}). The bold character marks the lowest MSE among all methods for each $Q$. The MMD estimator performs poorly only when there are outliers on the sphere of radius $\sqrt{d}$ centered at $\mathbf{0}$, i.e. $Q=\delta_{\{\theta_c\}}$ with $\|\theta_0-\theta_c\|=\sqrt{d}$.
} 
\end{table}

Additionally, Figure \ref{table:gaussianeps} shows that the estimation error is linear with respect to the proportion of outliers $\varepsilon$ in dimension $d=10$ with a sample size $n=5000$. We chose the contamination $Q=\mathcal{N}(\mathbf{5},I_d)$, but this choice is not crucial and other choices of the contamination distribution would lead to the same results.

\begin{figure}[htbp]
  \centering
  \includegraphics[width=0.5\linewidth]{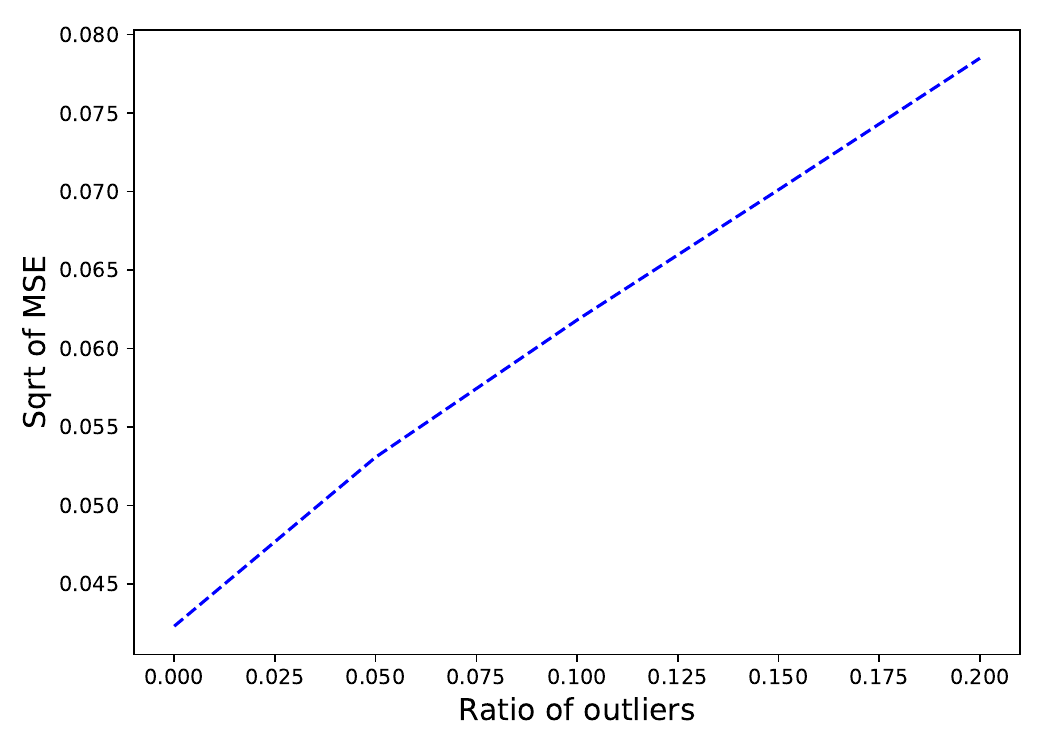}
  \vspace{-0.5cm}
  \caption{Mean square error as a function of the outliers ratio $\varepsilon$, for a dimension $d=10$, a sample size $n=5000$, and a Gaussian contamination $Q=\mathcal{N}(\mathbf{5},I_d)$. The error grows linearly as the ratio increases.}
   \label{table:gaussianeps} 
\end{figure}

Similarly, Figure \ref{table:gaussiand} shows the effect of the dimension on the estimation error, using $\varepsilon=0.1$ and $n=5000$. We can see that for a "harmless contamination" $Q=\mathcal{N}(\mathbf{5},I_d)$ (blue curve), there is no effect at all of the dimension on the estimation error. This result is still valid for other contaminations. At the opposite, when choosing the "worst-case contamination" $Q=\delta_{\{\mathbf{1}\}}$ (red curve), such that the distance between the true parameter and the corrupted mean is $\sqrt{d}$, then the estimation error grows linearly in the square root of the dimension, which explains the rate of convergence of the MMD estimator.

\begin{figure}[htbp]
 \centering
 \includegraphics[width=0.5\linewidth]{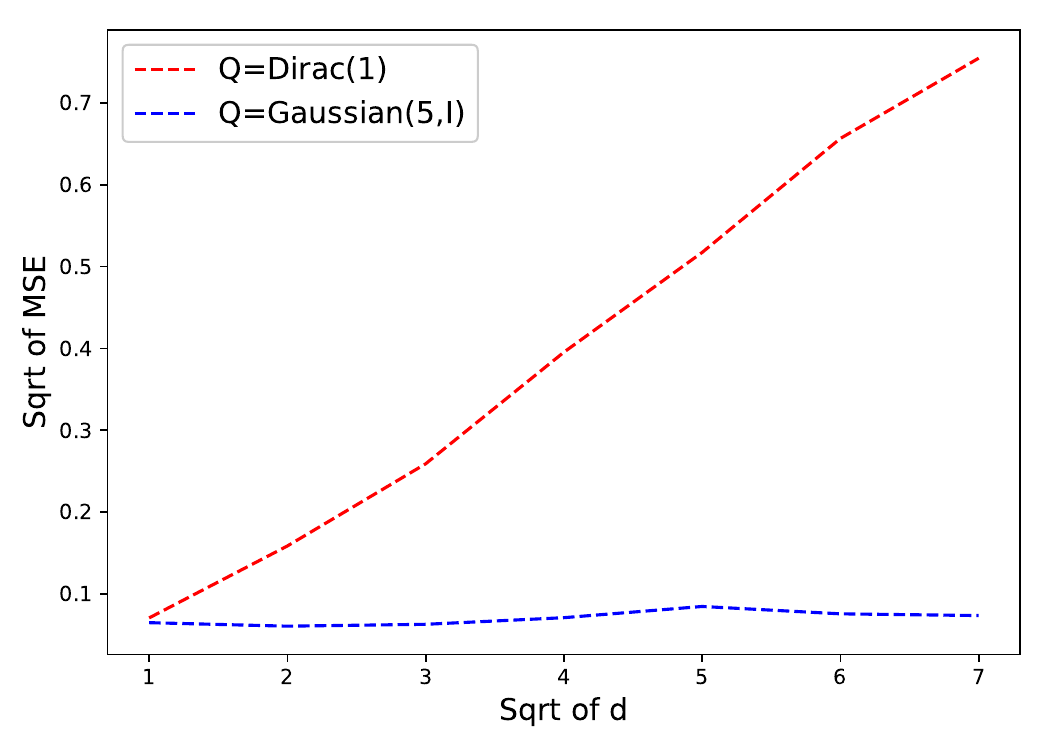}
  \vspace{-0.2cm}
 \caption{Mean square error as a function of the square root of the dimension $\sqrt{d}$, for an outlier ratio $\varepsilon=0.1$, a sample size $n=5000$, and two different contaminations: a "harmless" Gaussian $Q=\mathcal{N}(\mathbf{5},I_d)$ and a "worst-case" Dirac $Q=\delta_{\{\mathbf{1}\}}$. The error grows linearly in the Dirac case but is not affected by the dimension in the Gaussian case.}
 \label{table:gaussiand} 
\end{figure}

\vspace{-0.2cm}

\subsection{Gaussian mixture estimation}

In the second experiment of this paper, we sample data according to a three component Gaussian mixture $0.3. \mathcal{N}(-3.72,1) + 0.3 .\mathcal{N}(0.11,1) + 0.4 .\mathcal{N}(4.54,1)$. Here, we use the same approach than in Section \ref{dictionary}. We try to estimate the mixture as a linear combination of mixture in a dictionary composed of all Gaussians of variance 1 and whose means range from -5 to 5 with a stepsize of 0.02. Note that the Gaussian $\mathcal{N}(0.11,1)$ is not even in the dictionary. The goal is to estimate the weights of each Gaussian in the dictionary using MMD estimation. This estimation method is compared to the gold standard Expectation-Maximization (EM) \cite{EM} algorithm and to the tempered Coordinate Ascent Variational Inference (CAVI) algorithm \cite{BECA2018,blei2017variational} that estimate directly the means and the weights of the three-component mixture, using ten random initializations. The experiment is conducted first without any outlier, and then with an outlier equal to 100. Here, the metric is a mean average error (MAE) between the densities. First, we sample 10.000 datapoints independently according to the true mixture. Then, we evaluate the  difference between the true density $p^0$ and the estimated density $p_{\hat{\theta}_n}$ evaluated at each of the 10.000 datapoints, and we finally take the average:
$$
\left\{
\begin{array}{l}
  z_1,...,z_{N} \stackrel{i.i.d}{\sim} p^0 \hspace{0.1cm} \textnormal{where} \hspace{0.1cm} N=10.000 , \\
  \textnormal{MAE} =  \frac{1}{N} \sum\limits_{\ell=1}^{N} \left| p^0(z_\ell) - p_{\hat{\theta}_n}(z_\ell) \right| .
\end{array}
\right.
$$
Again, the final metric is the average over 50 repetitions of the experiment. Figures 4, 5 and 6, and Table 1 clearly show that our estimator performs comparably to both the EM and the CAVI algorithms in the well-specified case, while it is the only one that is not sensitive to the outlier and that gives a consistent estimate.

\begin{table}[ht]
 \centering 
\begin{tabular}{c c c} 
\hline\hline 
Algorithm & Without the outlier & With the outlier \\ [0.5ex] 
\hline 
\textbf{MMD} & \textbf{0.0170 (0.0052)} & \textbf{0.0173 (0.0045)} \\ 
CAVI & 0.0218 (0.0172) & 0.0976 (0.0002) \\
EM & 0.0186 (0.0147) & 0.0738 (0.0186) \\ [1ex] 
\hline 
\end{tabular}
\label{table:nonlin} 
\caption{MAE for the Gaussian mixture with/without the outlier (with the corresponding standard deviations)} 
\end{table}

\begin{figure}[htbp]
  \centering
  \includegraphics[width=1.0\linewidth]{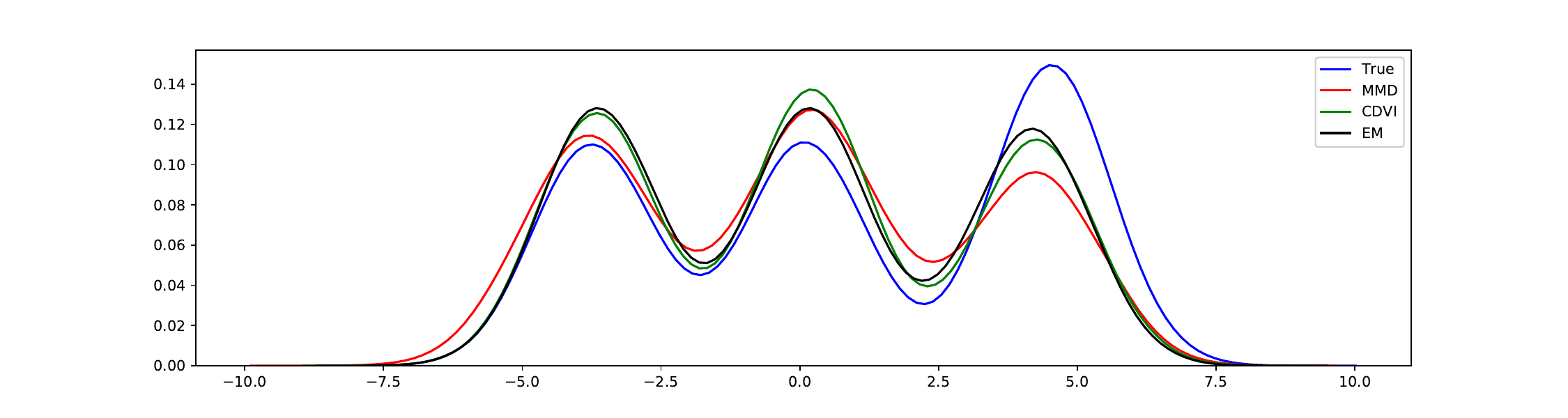}
  \vspace{-0.7cm}
  \caption{Plot of the estimated densities using different methods without outliers. The blue curve represents the true density, the red one the MMD density, the green one the CAVI density and the black one the EM density.}
\end{figure}

\begin{figure}[htbp]
  \includegraphics[width=1.0\linewidth]{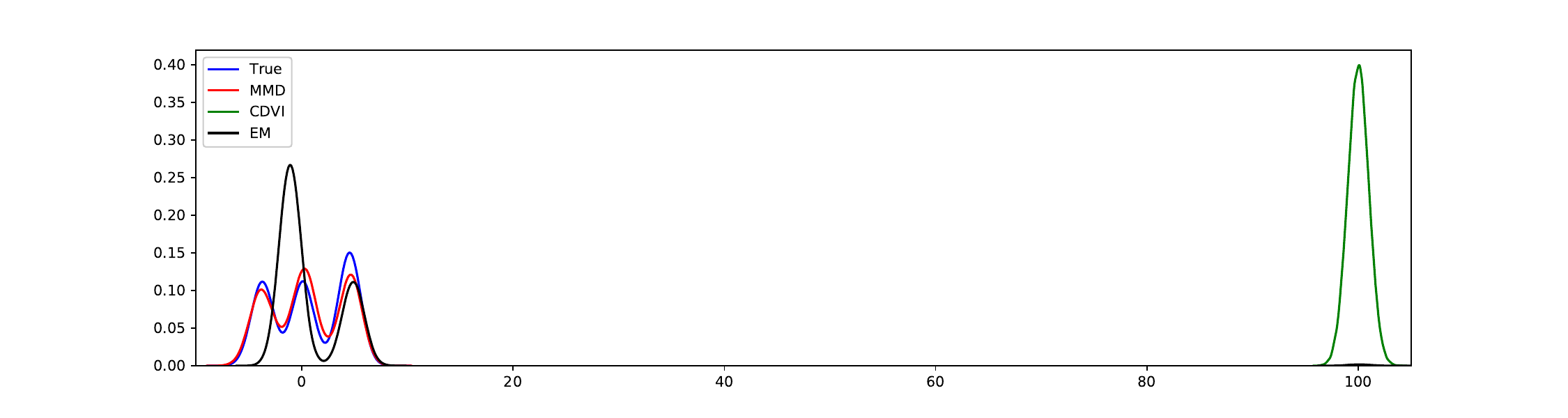}
  \vspace{-0.7cm}
  \caption{Plot of the estimated densities using different methods in presence of 1 outlier at 100. The blue curve represents the true density, the red one the MMD density, the green one the CAVI density and the black one the EM density. The EM estimate has a small component at 100, and CAVI only one component at 100.}
\end{figure}

\begin{figure}[htbp]
  \includegraphics[width=1.0\linewidth]{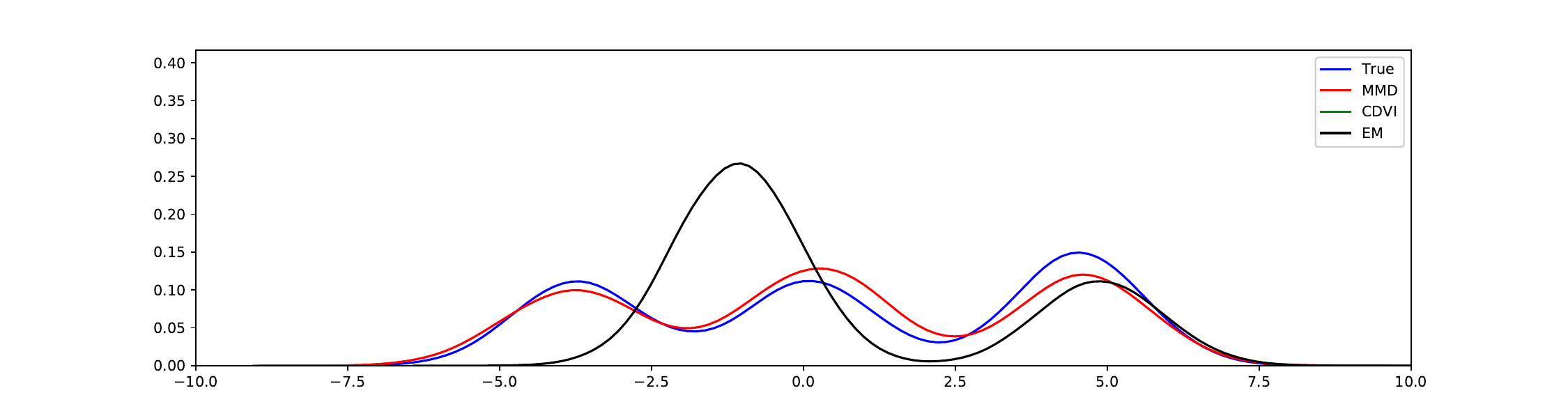}
  \vspace{-0.7cm}
  \caption{Zoom of Figure 5, without the component of EM at 100.}
\end{figure}

\section{Proofs of the main theorems}
\label{sec:proofs}

\subsection{A preliminary lemma: convergence of $\hat{P}_n$ to $P^0$ with respect to $\mathbb{D}_k$}

\begin{lemma}
 \label{lemma:mmd:1}
 We have
 $$ \mathbb{E} \left[ \mathbb{D}_k^2\left( \hat{P}_n,P^0 \right) \right] \leq \frac{  1 + 2 \sum_{t=1}^n \left(1-\frac{t}{n}\right)\varrho_t }{n}. $$
\end{lemma}

\begin{proof}
First, note that $\mathbb{E} ( \left\| k(X_i,\cdot) - \mu_{P^0} \right\|_{{\mathcal{H}_{k}}}^2) \leq \mathbb{E} (\left\| k(X_i,\cdot) \right\|_{{\mathcal{H}_{k}}}^2) $. This formula is the RKHS version of ${\rm Var}(X) \leq \mathbb{E}(X^2)$ and is proven in the following way:
\begin{align}
 \mathbb{E} \left( \left\| k(X_i,\cdot) - \mu_{P^0} \right\|_{{\mathcal{H}_{k}}}^2 \right)
 & =  \mathbb{E} \left( \left\| k(X_i,\cdot) \right\|_{{\mathcal{H}_{k}}}^2 \right) - 2 \mathbb{E} \left( \left< k(X_i,\cdot) , \mu_{P^0}  \right>_{{\mathcal{H}_{k}}}^2 \right) +  \mathbb{E} \left( \left\| \mu_{P^0} \right\|_{{\mathcal{H}_{k}}}^2 \right)
 \nonumber
 \\
 & = \mathbb{E} \left( \left\| k(X_i,\cdot) \right\|_{{\mathcal{H}_{k}}}^2 \right) - 2  \left< \mathbb{E} \left( k(X_i,\cdot) \right), \mu_{P^0}  \right>_{{\mathcal{H}_{k}}}^2 +  \mathbb{E} \left( \left\| \mu_{P^0} \right\|_{{\mathcal{H}_{k}}}^2 \right)
 \nonumber
 \\
 & = \mathbb{E} \left( \left\| k(X_i,\cdot) \right\|_{{\mathcal{H}_{k}}}^2 \right) -  \mathbb{E} \left( \left\| \mu_{P^0} \right\|_{{\mathcal{H}_{k}}}^2 \right) \leq \mathbb{E} \left( \left\| k(X_i,\cdot) \right\|_{{\mathcal{H}_{k}}}^2 \right).
 \label{var.RKHS}
\end{align}
Let us now prove the lemma. We have
\begin{align*}
\mathbb{E} \left[ \mathbb{D}_k^2\left( \hat{P}_n,P^0 \right) \right] 
& = \mathbb{E} \left\{  \left\| \frac{1}{n}\sum_{i=1}^{n}\left[ k(X_i,\cdot) - \mu_{P^0} \right] \right\|_{{\mathcal{H}_{k}}}^2 \right\}
\\
& = \frac{1}{n^2}  \mathbb{E} \left\{ \sum_{i=1}^n \left\| k(X_i,\cdot) - \mu_{P^0} \right\|_{{\mathcal{H}_{k}}}^2  + 2 \sum_{1\leq i<j \leq n} \left<k(X_i,\cdot) - \mu_{P^0},k(X_j,\cdot) -\mu_{P^0} \right>_{\mathcal{H}_k} \right\}.
\end{align*}
We upper bound the first term in the right-hand side thanks to~\eqref{var.RKHS}, while we remark that the second term exactly matches the definition of $\varrho_{|i-j|}$. This leads to:
\begin{align*}
\mathbb{E} \left[ \mathbb{D}_k^2\left( \hat{P}_n,P^0 \right) \right] 
& \leq \frac{1}{n^2}  \mathbb{E} \left\{ \sum_{i=1}^n \left\| k(X_i,\cdot) \right\|_{{\mathcal{H}_{k}}}^2  +  2 \sum_{1\leq i<j \leq n} \varrho_{|i-j|} \right\}
\\
& = \frac{1}{n^2}  \mathbb{E} \left\{ \sum_{i=1}^n k(X_i,X_i)   +  2 \sum_{1\leq i<j \leq n} \varrho_{|i-j|} \right\}.
\end{align*}
As we assumed in all the paper that $|k(x,x')| \leq 1$ for any $(x,x')\in\mathbb{X}^2$, we obtain:
$$
\mathbb{E} \left[ \mathbb{D}_k^2\left( \hat{P}_n,P^0 \right) \right] 
 \leq \frac{1}{n^2} \left( n  + 2 \sum_{1\leq i<j \leq n} \varrho_{|i-j|}\right)
 = \frac{  1 + 2 \sum_{t=1}^n \left(1-\frac{t}{n}\right)\varrho_t }{n}.
$$
\end{proof}
Note that in the i.i.d case, this leads to
$$  \mathbb{E} \left[ \mathbb{D}_k^2\left( \hat{P}_n,P^0 \right) \right] \leq \frac{1}{n} $$
and thus
$$  \mathbb{E} \left[ \mathbb{D}_k\left( \hat{P}_n,P^0 \right) \right] \leq \sqrt{ \mathbb{E} \left[ \mathbb{D}_k^2\left( \hat{P}_n,P^0 \right) \right]} \leq \frac{1}{\sqrt{n}} .$$
The rate $1/\sqrt{n}$ is known to be minimax in this case: Theorem 1 in~\cite{Tol2017}.

\subsection{Proof of Theorem~\ref{theorem:mmd:1}}

\begin{proof}
 First,
 $$ \mathbb{D}_k\left(P_{\hat{\theta}_n},P^0 \right) \leq \mathbb{D}_k\left(P_{\hat{\theta}_n},\hat{P}_n \right) + \mathbb{D}_k\left(\hat{P}_n,P^0 \right) \leq \mathbb{D}_k\left(P_{\theta},\hat{P}_n \right) + \mathbb{D}_k\left(\hat{P}_n,P^0 \right)   $$
 for any $\theta\in\Theta$, by definition of $\hat{\theta}_n$, and thus, using the triangle inequality again,
 $$ \mathbb{D}_k\left(P_{\hat{\theta}_n},P^0 \right)  \leq \mathbb{D}_k \left(P_{\theta},P^0 \right) + 2 \mathbb{D}_k\left(\hat{P}_n,P^0 \right).   $$
 Take the expectation on both sides and note that
 $$ \mathbb{E} \left[\mathbb{D}_k\left(\hat{P}_n,P^0 \right)\right] \leq \sqrt{ \mathbb{E} \left[\mathbb{D}_k^2\left(\hat{P}_n,P^0 \right)\right]} \leq \sqrt{\frac{  1 + 2 \sum_{t=1}^n \left(1-\frac{t}{n}\right)\varrho_t }{n}} \leq \sqrt{\frac{  1 + 2 \sum_{t=1}^n \varrho_t }{n}}  $$
 where the second inequality is given by Lemma~\ref{lemma:mmd:1}.
\end{proof}

\subsection{Proof of Theorem~\ref{theorem:mmd:briol:improved}}

We start by reminding the following result from~\cite{Briol2019}; similar results can be found in~\cite{Song2008} or~\cite{Gretton2009}.
\begin{lemma}[Lemma 1 page 10~\cite{Briol2019}]
 \label{lemma:mmd:briol}
 For any $\delta>0$,
 $$ \mathbb{P} \left( \mathbb{D}_k\left( \hat{P}_n,P^0 \right) \leq \frac{1}{\sqrt{n}} \left(1+\sqrt{\log\left(1/\delta\right)} \right) \right) \geq 1-\delta.$$
\end{lemma}
This result (that we won't use here) relies on McDiarmid inequality \cite{McDo} who proposed a beautiful way to control the difference between a function of the data, $f(X_1,\dots,X_n)$, and its expectation. The idea relies on writing this function as a martingale, $f(X_1,\dots,X_n)=M_n$ where $M_t$, for $t\leq n$, is given by $M_t = \mathbb{E}[f(X_1,\dots,X_n)|X_{1},\dots,X_t]$, and controling the martingale increments. It appears that many inequalities can be proven by using this technique, this is discussed in details in Chapter 3 in~\cite{BoucheronConcentration}. Using this technique, Rio~\cite{McDoRIO} proved a version of McDiarmid's inequality for series satisfying Assumption~\ref{asm:gamma:mixing} (note that the paper is written in French, a more recent paper by the same author~\cite{McDoRIO2} in English contains this result and new ones). We start by reminding Rio's result.
\begin{lemma}[Theorem 1 page~906~\cite{McDoRIO}]
\label{lemma:rio} Under Assumption~\ref{asm:gamma:mixing},
assume that $\mathcal{H}_k^n \rightarrow \mathbb{R}$ satisfies:
$$ \Bigl| f(a_1,\dots,a_n) - f(a_1',\dots,a_n') \Bigr| \leq \sum_{i=1}^n \|a_i-a_i'\|_{\mathcal{H}_k}.$$
Then, for any $t >0$,
$$ \mathbb{E}\exp\biggl[ t f(\mu_{\delta_{X_1}},\dots,\mu_{\delta_{X_1}}) - t \mathbb{E}[f(\mu_{\delta_{X_1}},\dots,\mu_{\delta_{X_1}})]  \biggr] \leq \exp\left( \frac{t^2 (1+\Gamma)^2 n}{2} \right). $$
\end{lemma}
This allows us to state our variant of Lemma~\ref{lemma:mmd:briol}.
\begin{lemma}
 \label{lemma:mmd:nous}
 Under Assumptions~\ref{asm:our:mixing} and~\ref{asm:gamma:mixing},
 $$ \mathbb{P}\left(  \mathbb{D}_k\left( \hat{P}_n ,P^0 \right) \leq \frac{\sqrt{1+2\Sigma} + (1+\Gamma)\sqrt{2\log\left(\frac{1}{\delta}\right)} }{\sqrt{n}} \right) \geq 1-\delta. $$
\end{lemma}
\begin{proof}[Proof of Lemma~\ref{lemma:mmd:nous}]
Define
$$ f(a_1,\dots,a_n) = \left\|\sum_{i=1}^n (a_i - \mu_{P^0}) \right\|_{\mathcal{H}_k}. $$
For any $x>0$ and any $t>0$,
\begin{align*}
 \mathbb{P}\Biggl( & \mathbb{D}_k\left( \hat{P}_n ,P^0 \right) - \mathbb{E}[\mathbb{D}_k\left( \hat{P}_n ,P^0 \right) ] \geq x \Biggr)
 \\
 & =  \mathbb{P}\left( \frac{f(\mu_{\delta_{X_1}},\dots,\mu_{\delta_{X_n}})}{n}  - \mathbb{E}\left[\mathbb{D}_k\left( \hat{P}_n ,P^0 \right) \right] \geq x \right)
 \\
 & \leq  \mathbb{E}\exp \left( t \frac{f(\mu_{\delta_{X_1}},\dots,\mu_{\delta_{X_n}})}{n}  -  t \mathbb{E}\left[\mathbb{D}_k\left( \hat{P}_n ,P^0 \right) \right] - t x \right) \text{ by Markov inequality}
 \\
 & \leq \exp\left( \frac{t^2 (1+\Gamma)^2 }{2n} - tx \right) \text{ by Lemma~\ref{lemma:rio}}
 \\
 & =  \exp\left( - \frac{x^2 n}{2 (1+\Gamma)^2 }  \right)
\end{align*}
where we chose $t=xn/(1+\Gamma)^2$. Put $x = (1+\Gamma)\sqrt{2\log(1/\delta)/n}$ to get:
$$ \mathbb{P}\left(  \mathbb{D}_k\left( \hat{P}_n ,P^0 \right) \leq  \mathbb{E}[\mathbb{D}_k\left( \hat{P}_n ,P^0 \right) ] + (1+\Gamma) \sqrt{\frac{2\log\left(\frac{1}{\delta}\right)}{n}} \right) \geq 1-\delta. $$
Use Theorem~\ref{lemma:mmd:1} to upper bound the expectation in the right-hand side. This gives the claimed result:
$$ \mathbb{P}\left(  \mathbb{D}_k\left( \hat{P}_n ,P^0 \right) \leq \sqrt{\frac{1+2\Sigma}{n}} + (1+\Gamma) \sqrt{\frac{2\log\left(\frac{1}{\delta}\right)}{n}} \right) \geq 1-\delta. $$
\end{proof}

We are now in position to prove Theorem~\ref{theorem:mmd:briol:improved}.
\begin{proof}[Proof of Theorem~\ref{theorem:mmd:briol:improved}]
 With probability $1-\delta$, for any $\theta\in\Theta$,
 \begin{align*}
  \mathbb{D}_k\left( P_{\hat{\theta}_n},P^0 \right)
  & \leq \mathbb{D}_k\left( P_{\hat{\theta}_n},\hat{P}_n \right)
         + \mathbb{D}_k\left( \hat{P}_n ,P^0 \right) \text{ (triangle inequality)}
  \\
  & \leq  \mathbb{D}_k\left( P_{\theta},\hat{P}_n \right) + \frac{\sqrt{1+2\Sigma} + (1+\Gamma)\sqrt{2\log\left(\frac{1}{\delta}\right)} }{\sqrt{n}} \text{ (definition of $\hat{\theta}_n$ and Lemma~\ref{lemma:mmd:nous})}
  \\
  & \leq  \mathbb{D}_k\left( P_{\theta},P^0 \right) + \mathbb{D}_k\left( \hat{P}_n ,P^0 \right) + \frac{\sqrt{1+2\Sigma} + (1+\Gamma)\sqrt{2\log\left(\frac{1}{\delta}\right)} }{\sqrt{n}}\text{ (triangle inequality)}
  \\
  & \leq  \mathbb{D}_k\left( P_{\theta},P^0 \right)
        +  2 \frac{\sqrt{1+2\Sigma} + (1+\Gamma)\sqrt{2\log\left(\frac{1}{\delta}\right)} }{\sqrt{n}}\text{ (Lemma~\ref{lemma:mmd:nous})}
 \end{align*}
\end{proof}

\subsection{Proof of Lemma~\ref{lemma:huber} and of Proposition~\ref{prop:adversarial}}

\begin{proof}[Proof of Lemma~\ref{lemma:huber}]
We have
\begin{align*}
\left| \mathbb{D}_k(P_{\theta},P^0) - \mathbb{D}_k (P_{\theta},P_{\theta_0}) \right|
& \leq  \mathbb{D}_k(P^0,P_{\theta_0})
\\
& = \left\| (1-\varepsilon)\mu_{P_{\theta_0}} + \varepsilon \mu_Q - \mu_{P_{\theta_0}} \right\|_{{\mathcal{H}_{k}}} \text{ (definition of $P^0$)}
\\
& =  \left\| \varepsilon (\mu_Q -\mu_{P_{\theta_0}}) \right\|_{{\mathcal{H}_{k}}}
\\
& \leq \varepsilon (\left\|\mu_Q \right\|_{{\mathcal{H}_{k}}} + \left\|\mu_{P_{\theta_0}} \right\|_{{\mathcal{H}_{k}}}) \text{ (triangle inequality)}
\\
& \leq  2 \varepsilon.
\end{align*}
\end{proof}

\begin{proof}[Proof of Proposition~\ref{prop:adversarial}]
 Let us put
 $$ \tilde{P}_n = \frac{1}{n}\sum_{i=1}^n \delta_{\tilde{X}_i}. $$
First, note that for any probability measure $Q$,
\begin{align}
   \left| \mathbb{D}_k(Q,\tilde{P}_n) - \mathbb{D}_k(Q,\hat{P}_n) \right|
   & \leq \mathbb{D}_k(\hat{P_n},\tilde{P}_n)
   \nonumber
   \\
   & =  \left\|\frac{1}{n}\sum_{i=1}^n \left( k(X_i,\cdot)- k(\tilde{X}_i,\cdot) \right)\right\|_{{\mathcal{H}_{k}}}
   \nonumber
   \\
   & \leq  \frac{1}{n}\sum_{i=1}^n \left\| k(X_i,\cdot)- k(\tilde{X}_i,\cdot) \right\|_{{\mathcal{H}_{k}}}
   \nonumber
   \\
   & = \frac{1}{n}\sum_{i\in\mathcal{O}} \left\| k(X_i,\cdot)-k(\tilde{X}_i,\cdot) \right\|_{{\mathcal{H}_{k}}}
   \nonumber
   \\
   & \leq \frac{2|\mathcal{O}|}{n}
   \nonumber
   \\
   & \leq 2\varepsilon.
   \label{proof.adversarial.1}
\end{align}
Consider $Q = P_{\tilde{\theta}_n}$. Then:
\begin{align*}
 \mathbb{D}_k(P_{\tilde{\theta}_n},P^0)
 & \leq  \mathbb{D}_k(P_{\tilde{\theta}_n},\tilde{P}_n) + \mathbb{D}_k(\tilde{P}_n,P^0)
 \\
 & \leq \mathbb{D}_k(P_{\hat{\theta}_n},\tilde{P}_n) + \mathbb{D}_k(\tilde{P}_n,P^0) \text{ by definition of } \tilde{\theta}_n
 \\
 & \leq \left[2\varepsilon + \mathbb{D}_k(P_{\hat{\theta}_n},\hat{P}_n)\right]  + \left[2\varepsilon + \mathbb{D}_k(\hat{P_n},P^0)\right]
\end{align*}
where we used~\eqref{proof.adversarial.1} with $Q = P_{\hat{\theta}_n}$ and then $Q = P^0$ respectively. So:
\begin{align*}
 \mathbb{D}_k(P_{\tilde{\theta}_n},P^0)
 & \leq 4\varepsilon +  \mathbb{D}_k(P_{\hat{\theta}_n},\hat{P}_n) + \mathbb{D}_k(\hat{P_n},P^0)
 \\
 & \leq 4\varepsilon +  \mathbb{D}_k(P_{\theta_0},\hat{P}_n) + \mathbb{D}_k(\hat{P_n},P^0)  \text{ by definition of } \hat{\theta}_n
 \\
 & = 4\varepsilon + 2 \mathbb{D}_k(\hat{P_n},P^0)
\end{align*}
as it is here assumed that $P^0 =P_{\theta_0}$. 
\end{proof}

\section{Conclusion}

Parametric estimation with MMD provides a simple way to define universally consistent, robust estimators. In many settings, these estimators also have optimal rates of convergence. The computation of the MMD-based estimator can generally be done through a stochastic gradient descent. We thus believe that it is a practically reasonable and nice alternative to many robust estimation procedures.

Interestingly, Proposition~\ref{prop:ex:gauss} provides a natural calibration to the kernel parameter, which is usually a problem in practice. However, in more general settings, the calibration of this parameter, and the choice of the kernel, remain important open questions.

The application of this method to more sophisticated models in statistics and in machine learning (time series models, regression) should be investigated in details and will be the object of future works. In particular, the coefficients $\varrho_t$ of Definition~\ref{dfn.varrho} are new to our knowledge and it would be interesting to compare them to more weak dependence coefficients\footnote{In a personal communication, George Wynne (Imperial College London) proved a connection to yet another notion of dependence: $L^p$-$m$-approximability, used for functional time series~\cite{Hormann2010}. He proved that if $(X_t)_{t\in\mathbb{Z}}$ is $L^p$-$m$-approximable then $\sum_{t=1}^\infty \varrho_t < +\infty$.}.

\section*{Acknowledgements}
 
We would like to thank Guillaume Lecu\'e (ENSAE Paris) for his helpful comments, Mathieu Gerber (University of Bristol) who fixed a mistake in the constants in Proposition~\ref{prop:ex:gauss}, and George Wynne (Imperial College) for his very informative comments on the coefficients $\varrho_t$. We also would like to thank the anonymous Referees and the Associate Editor for their insightful comments that helped to improve the structure of the paper. All the remaining mistakes are ours.

\bibliographystyle{apalike}

\newpage

\section*{Supplementary material}

\subsection{Proofs of Section~\ref{sec:examples}}

\begin{proof}[Proof of Proposition~\ref{prop:ex:gauss}]
 We remind that $P_\theta = \mathcal{N}(\theta,\sigma^2 I_d)$ where $\theta\in\Theta=\mathbb{R}^d$.
When $X$ and $Y$ are independent, respectively from $P_\theta$ and $P_{\theta'}$, we have $(X-Y)\sim \mathcal{N}(\theta-\theta',\sigma^2 I_d)$. Thus,
$$ \frac{(X-Y)}{\sqrt{2\sigma^2}} \sim \mathcal{N}\left( \frac{(\theta-\theta')}{\sqrt{2\sigma^2 }},I_d \right) $$
and thus the square of this random variable is a noncentral chi-square random variable:
$$ \frac{\| X-Y \|^2}{2\sigma^2} \sim \chi^2\left( d,  \frac{\| \theta-\theta' \|^2}{2\sigma^2} \right). $$
It is known that when $U\sim \chi^2(d,m)$ we have $ \mathbb{E}[\exp(tU)] = \exp(mt/(1-2t)) /(1-2t)^{d/2} $. Taking $t=-(2\sigma^2)/\gamma^2$, this leads to
\begin{equation*}
\left<\mu_{P_\theta},\mu_{P_{\theta'}} \right>_{\mathcal{H}} = \mathbb{E}_{X\sim P_\theta,Y\sim P_{\theta'}}\left[\exp\left(-\frac{\|X-Y\|^2}{\gamma^2}\right)\right]
= \left(\frac{\gamma^2}{4\sigma^2 + \gamma^2}\right)^{\frac{d}{2}} \exp\left( -\frac{\|\theta-\theta'\|^2}{4\sigma^2 + \gamma^2} \right)
\end{equation*}
and thus
\begin{equation}
\label{eq:mmd:gaussian}
\mathbb{D}_{k_\gamma}^2(P_\theta,P_{\theta'}) 
= 2 \left(\frac{\gamma^2}{4\sigma^2 + \gamma^2}\right)^{\frac{d}{2}} \left[ 1 - \exp\left( -\frac{\|\theta-\theta'\|^2}{4\sigma^2 + \gamma^2} \right) \right].
\end{equation}
From~\eqref{eq:mmd:gaussian} and Proposition~\ref{prop:adversarial}, we obtain, with probability at least $1-\delta$,
  $$ 
  2 \left(\frac{\gamma^2}{4\sigma^2 + \gamma^2}\right)^{\frac{d}{2}} \left[ 1 - \exp\left( -\frac{\|\theta-\tilde{\theta}_n\|^2}{4\sigma^2 + \gamma^2} \right) \right]
  = \mathbb{D}^2_k\left( P_{\tilde{\theta}_n},P_{\theta_0}\right) \leq  16 \left( \varepsilon 
 +  \frac{1 + \sqrt{2\log\left(\frac{1}{\delta}\right)} }{\sqrt{n}} \right)^2,
$$
that is,
$$
\|\theta-\tilde{\theta}_n\|^2
\leq - (4\sigma^2 + \gamma^2)
 \log\left\{ 1-8 \left(\frac{4\sigma^2 + \gamma^2}{\gamma^2}\right)^{\frac{d}{2}} \left( \varepsilon 
 +  \frac{1 + \sqrt{2\log\left(\frac{1}{\delta}\right)} }{\sqrt{n}} \right)^2 \right\}.
$$
Now, use inside the $\log$ the inequality $(1+1/x)^x \leq {\rm e}$ on $x=\gamma^2/(4\sigma^2)$ to get:
$$
\|\theta-\tilde{\theta}_n\|^2
\leq - (4\sigma^2 + \gamma^2)
 \log\left\{ 1-8 {\rm e}^{\frac{2\sigma^2 d}{\gamma^2}} \left( \varepsilon 
 +  \frac{1 + \sqrt{2\log\left(\frac{1}{\delta}\right)} }{\sqrt{n}} \right)^2 \right\}.
$$
This proves~\eqref{ex:gauss:3}. Simply plug $\gamma=\sigma \sqrt{2d}$ to get the second inequality.
\end{proof}

The Gaussian example appears to be a very special case, where it is possible to derive an explicit formula for the MMD distance. In most models, this is not possible. However, in many models, we observed that it is actually possible to provide an explicit formula for the $L^2$ distance, of the form $ \| \theta-\theta' \|^2 = F( \|p_{\theta} - p_{\theta'} \|_{L^2}^2 ) $, where  $P_\theta$ has a density $p_\theta\in L^2$ with respect to the Lebesgue measure and $F$ is a nondecreasing function. It is then tempting to use the connection between the MMD distance and the $L^2$ distance mentioned in Remark~\ref{rmk:L2} and in~\cite{L2}. The scheme of the proof is as follows:
$$ \| \hat{\theta} - \theta_0 \|^2 =
 F\left( \|p_{\hat{\theta}} - p_{\theta_0} \|^2_{L^2} \right)
 \underbrace{\leq}_{?}
 F\left( c \mathbb{D}_k^2 (P_{\hat{\theta}},P_{\theta_0}) \right)
 \underbrace{\leq}_{
\begin{tiny}
\begin{array}{c}
 \text{Theorem~\ref{theorem:mmd:1}}
 \\
 \text{Theorem~\ref{theorem:mmd:briol:improved}}
 \\
 \text{Proposition~\ref{prop:adversarial}}
 \\
 ...
\end{array}
\end{tiny}
}
F(\text{bound})
 $$
which means that in such models, the only new step to get a rate of convergence on $\hat{\theta}$ is to check the condition
\begin{equation}
\label{cond:L2}
 \|p_{\theta} - p_{\theta'} \|_{L^2}^2 \leq c \mathbb{D}_k^2 (P_{\theta},P_{\theta'})
\end{equation}
for some $c>0$. Finally, in any case where $k_\gamma(x,x') = K (\|x-x'\| / \gamma)$, let $\mu$ denote the Fourier transform of $K$:
$$ \mu(t) := \mathcal{F}[K](t)
=  \int K(x) \exp(-2i\pi \left<t,x\right>) {\rm d}x  .$$
Note that for the Gaussian kernel in $\mathbb{R}^d$, $K(u) = \exp(-\|u\|^2)$, $\mu(t) = \pi^{d/2} \exp(-\|t\|^2/4) $.
We remind a few properties of the Fourier transform. First,
\begin{equation}
\label{prop.fourier.1}
\mathcal{F}[K(\cdot/\gamma)](t) = \gamma^d \mathcal{F}[K](\gamma t) = \gamma^d \mu(\gamma t).
\end{equation}
Let $\star$ denote the convolution product:
$$ p\star q(x) = \int p(x-t) q(t) {\rm d}t $$
and we remind the classical result
\begin{equation}
\label{prop.fourier.2}
\mathcal{F}[p\star q] = \mathcal{F}[p] \mathcal{F}[q].
\end{equation}
Finally, we remind that
\begin{equation}
\label{prop.fourier.3}
\int p(x) \overline{q(x)}{\rm d}x = \int \mathcal{F}[p](t) \overline{\mathcal{F}[q](t)} {\rm d}t .
\end{equation}
Then, we have:
 \begin{align*}
 \mathbb{D}_{k_\gamma}^2(P_\theta,P_{\theta'})
 & = \iint K(\|x-y\|/\gamma) [p_\theta(x)-p_{\theta'}(x)][p_\theta(y)-p_{\theta'}(y)]{\rm d}x {\rm d}y
 \\
 & = \int [K(\cdot/\gamma) \star (p_\theta-p_{\theta'})](y)[p_\theta(y)-p_{\theta'}(y)] {\rm d}y \text{, by definition of } \star
 \\
 & = \int [K(\cdot/\gamma) \star (p_\theta-p_{\theta'})](y)\overline{[p_\theta(y)-p_{\theta'}(y)]} {\rm d}y \text{, (densities are real-valued) }
 \\
 & = \int \mathcal{F}[K(\cdot/\gamma)\star (p_\theta-p_{\theta'})](t) \overline{\mathcal{F}[p_\theta-p_{\theta'}](t)} {\rm d}t \text{, by~\eqref{prop.fourier.3}}
 \\
 & = \int \mathcal{F}[K(\cdot/\gamma)] \mathcal{F}[p_\theta-p_{\theta'}](t) \overline{\mathcal{F}[p_\theta-p_{\theta'}](t)} {\rm d}t \text{, by~\eqref{prop.fourier.2}}
 \\
 & = \int \gamma^d \mu(\gamma t) \left| \mathcal{F}[p_\theta-p_{\theta'}](t)\right|^2  {\rm d}t
 \text{, by~\eqref{prop.fourier.1}}.
\end{align*}
So, an alternative way to check~\eqref{cond:L2} is to check that:
$$
\|p_{\theta} - p_{\theta'} \|_{L^2}^2 \leq c  \int \gamma^d \mu(\gamma t) \left| \mathcal{F}[p_\theta-p_{\theta'}](t)\right|^2  {\rm d}t.
$$

\begin{proof}[Proof of Proposition~\ref{prop:ex:cauchy}] We have
$$ \left<p_\theta,p_{\theta'}\right>_{L^2} = \frac{2}{\pi[(\theta-\theta')^2+4]} \Rightarrow \|p_\theta - p_{\theta'} \|_{L^2}^2 = \frac{1}{\pi}\left(1-\frac{1}{\frac{(\theta-\theta')^2}{4}+1}\right). $$
We now use the above remarks and check~\eqref{cond:L2}. Note that
$$ \mathcal{F}[p_\theta-p_{\theta'}](t) = [\exp(-it\theta)-\exp(-it\theta')]\exp(-|t|) $$
and so
\begin{align*}
\frac{\mathbb{D}^2_{k_\gamma}(P_\theta,P_\theta')}{\|p_\theta-p_\theta'\|_{L^2}^2}
& = \frac{\pi \int \gamma \mu(\gamma t) \left|[\exp(-it\theta)-\exp(-it\theta')]\exp(-|t|)\right|^2 }{1-\frac{1}{\frac{(\theta-\theta')^2}{4}+1}}
\\
& \geq \frac{\pi \int \gamma \mu(\gamma t) \left|[\exp(-it\theta)-\exp(-it\theta')]\exp(-t^2 -1)\right|^2 }{1-\frac{1}{\frac{(\theta-\theta')^2}{4}+1}}
\end{align*}
thanks to $|x| \leq x^2+1$. Identify the Fourier transform of the Gaussian density on the numerator to get
\begin{align*}
\frac{\mathbb{D}^2_{k_\gamma}(P_\theta,P_\theta')}{\|p_\theta-p_\theta'\|_{L^2}^2}
& = \frac{2 \pi \left[1-\exp\left(-\frac{\|\theta-\theta'\|^2}{2(\gamma+2)}\right)\right] }{\exp(2)\sqrt{1+\frac{2}{\gamma}}\left(1-\frac{1}{\frac{(\theta-\theta')^2}{4}+1}\right)}
\\
& = \frac{ 2 \pi  \left[1-\exp\left(-\frac{\|\theta-\theta'\|^2}{8}\right)\right] }{\exp(2) \sqrt{2} \left(1-\frac{1}{\frac{(\theta-\theta')^2}{4}+1}\right)}
\end{align*}
when $\gamma=2$. For $x>0$ we prove easily that
$$ \frac{1-\exp(-x/2) }{1-\frac{1}{1+x}} > \frac{1}{2}, $$
so
$$
\frac{\mathbb{D}^2_{k_\gamma}(P_\theta,P_\theta')}{\|p_\theta-p_\theta'\|_{L^2}^2}
 \geq \frac{ \pi }{\sqrt{2} \exp(2)   }\geq \frac{1}{4}.
$$
Now, in the adversarial contamination case,
\begin{align*}
(\hat{\theta}_n - \theta_0 )^2
& = 4\left[ 1 - \frac{1}{1-\pi \|p_{\tilde{\theta}}-p_{\theta_0}\| ^2 } \right]
\\
& \leq 4\left[ 1 - \frac{1}{1-4 \pi \mathbb{D}^2_{k_{\gamma}}(P_{\tilde{\theta}},P_{\theta_0}) } \right]
\\
& \leq 4\left[ 1 - \frac{1}{1-128 \pi \left( \varepsilon^2 + \frac{2 + 4\log(1/\delta) }{n} \right) } \right]
\end{align*}
where the last inequality comes from Proposition~\ref{prop:adversarial}.
\end{proof}

\begin{rmk}
We applied the technique to the translated Cauchy model. Note that it is possible to apply it to many other translation models. However, in the translated uniform model (not reported in this paper), it leads to a suboptimal rate of convergence: $1/\sqrt{n}$, while the MLE is is $1/n$. But in the simulations we made, we observed that the MMD works very well in the uniform model -- on the condition that $\gamma$ is taken very small. We thus wonder whether it is possible, with another proof technique, to prove that the MMD estimator with $\gamma = \gamma(n) \rightarrow 0$ when $n\rightarrow \infty$ reaches the optimal rate of convergence. This is still an open question.
\end{rmk}

\begin{proof}[Proof of Proposition~\ref{prop:beta}]
\textcolor{red}{
To clarify notations, assume that the process $(X_t)_{t\in\mathbb{Z}}$ is defined on some space $(\Omega,\mathcal{A},\mathbb{P})$. The image of $\mathbb{P}$ through $X_t(\cdot)$ is the marginal distribution of $X_t$, that is $P_t$. The stationarity assumption ensures $P_t=P_0$ for all $t\in\mathbb{Z}$. The image of $\mathbb{P}$ through $(X_0,X_t)$ is their joint distribution, that will be denoted by $P_{0:t}$.}

\textcolor{red}{
We can define another probability distribution $\mathbb{P}_{0 \otimes t}$ on $(\Omega\times\Omega,\sigma(X_0)\otimes \sigma(X_t))$ as the image of $\mathbb{P}$ through $i(\omega)=(\omega,\omega)$. In particular, for $A\in \sigma(X_0)$ and $B\in\sigma(X_t)$, we have $\mathbb{P}_{0 \otimes t}(A\times B)=\mathbb{P}(A\cap B)$. Let $\mathbb{P}_{0}$ denote the restriction of $\mathbb{P}$ to $\sigma(X_0)$ and $\mathbb{P}_{t}$ the restriction of $\mathbb{P}$ to $\sigma(X_t)$. In Section 1.6 in~\cite{mixing2}, it is stated that
$$
\beta(\sigma(X_0),\sigma(X_t)) = \sup_{C\in \sigma(X_0)\otimes \sigma(X_t)} | \mathbb{P}_{0 \otimes t}(C) - \mathbb{P}_{0}\otimes \mathbb{P}_{t} (C)|.
$$
In particular, consider, for a fixed $u>0$, the event: $C(u)=\{u\geq \|X_0-X_t\|\}$. While $C(u)$ does not belong to $ \sigma(X_0)\times \sigma(X_t)$, it can be obtained as a countable union of events in $ \sigma(X_0)\times \sigma(X_t)$, and thus $C(u)\in \sigma(X_0)\otimes \sigma(X_t)$. Thus,
\begin{align*}
\beta(\sigma(X_0),\sigma(X_t))
& \geq  | \mathbb{P}_{0 \otimes t}(C(u)) - \mathbb{P}_{0}\otimes \mathbb{P}_{t} (C(u))|
\\
& = | \mathbb{P}_{0 \otimes t}(\{ u\geq \|X_0-X_t\| \}) - \mathbb{P}_{0}\otimes \mathbb{P}_{t} (u\geq \|X_0-X_t\| )|
\\
& = | P_{0:t}(\{(x,y): u\geq \|x-y\| \}) - P_0\otimes P_t( \{ (x,y): u\geq \|x-y\| \})|.
\end{align*}
Now,
 \begin{align*}
 \varrho_t
 & = \left| \mathbb{E}\left< \mu_{\delta_{X_t}}-\mu_{P^0},\mu_{\delta_{X_0}}-\mu_{P^0} \right>_{\mathcal{H}_k} \right|
 \\
 & = \left| \int  k(x,y) P_{0:t}({\rm d}(x,y)) - \iint k(x,y) P^0({\rm d}x)  P^0({\rm d}y)\right|
 \\
 & =  \left| \int  \left(\int \mathbf{1}_{\{u\geq \|x-y\|\}} f(u) {\rm d}u\right) P_{0:t}({\rm d}(x,y)) - \iint  \left(\int \mathbf{1}_{\{u > \|x-y\|\}}\right) f(u) {\rm d}u P^0({\rm d}x)  P^0({\rm d}y)\right|
 \\
 & = \left| \int_{0}^\infty \left( \int   \mathbf{1}_{\{u \geq  \|x-y\|\}} P_{0:t}({\rm d}(x,y)) - \iint \mathbf{1}_{\{u \geq \|x-y\|\}} \ P^0({\rm d}x)  P^0({\rm d}y)  \right) f(u) {\rm d}u \right|
 \\
 & = \left| \int_{0}^\infty \left(  P_{0:t}(\{(x,y): u\geq \|x-y\| \}) - P_0\otimes P_t( \{ (x,y): u\geq \|x-y\| \}) \right) f(u) {\rm d}u \right|
 \\
 & \leq  \int_{0}^\infty \left|  P_{0:t}(\{(x,y): u\geq \|x-y\| \}) - P_0\otimes P_t( \{ (x,y): u\geq \|x-y\| \}) \right| f(u) {\rm d}u
 \\
 & \leq \int_{0}^\infty \beta(\sigma(X_0),\sigma(X_t)) f(u) {\rm d}u
 \\
 & = \beta(\sigma(X_0),\sigma(X_t))
\end{align*}
which ends the proof.}
\end{proof}

\begin{proof}[Proof of Proposition~\ref{prop:counterexample}]
We start by a few preliminary remarks on $(X_t)$. First, note that, for any $\ell$ and $t$,
$$ X_{\ell+t} = \underbrace{\varepsilon_{\ell+t} + A \varepsilon_{\ell+t-1} + \dots + A^{t-1} \varepsilon_{\ell + 1} }_{=: B_\ell^t} + A^t X_\ell $$
where $B_\ell^t$ is independent of $X_\ell$. Let $P_b^t$ denote the distribution of $B_\ell^{t}$ (it does not depend on $\ell$ because of the stationarity).

Also, note that,
 $$ X_1 = A X_0  + \varepsilon_1 $$
 and thus
 \begin{equation}
 \label{ar:norm} \|X_1\| \leq \|A\| \|X_0\| + \|\varepsilon_1\| .
 \end{equation}
Taking the expectation of~\eqref{ar:norm}, and using the stationarity, gives
 $$ \mathbb{E}(|X_0|) =  \mathbb{E}(|X_1|)  \leq \|A\| \mathbb{E}(\|X_0\|) + \mathbb{E}(\|\varepsilon_1\|) $$
 which leads to
 $$ \mathbb{E}(\|X_0\|) \leq \frac{\mathbb{E}(\|\varepsilon_1\|) }{1-|a|}  = \frac{\mathbb{E}(\|\varepsilon_0\|) }{1-\|a\|}.  $$
 In the same way, if $\|\varepsilon_t\|< c$ almost surely for any $t$, taking the supremum of~\eqref{ar:norm} and using the stationarity gives, almost surely,
 $$ \|X_0\| \leq \frac{c}{1-\|a\|}. $$

We are now ready to start the derivation of the upper bound for $\varrho_t$. We have
 \begin{align*}
 \varrho_t &  =  \left| \int  k(x,y) P_{0:t}({\rm d}(x,y)) - \iint k(x,y) P^0({\rm d}x)  P^0({\rm d}y)\right|
  \\
  & = \left| \iint  k(x,A^tx+b) P^{0}({\rm d}x) P_b^t ({\rm d}b) - \iiint k(x,A^t x'+b) P^0({\rm d}x) P^0({\rm d}x') P_b^t({\rm d}b)\right|
  \\
  & \leq \iiint \left| k(x,A^t  x+b) - k(x,A^t x'+b) \right| P^0({\rm d}x) P^0({\rm d}x') P_b^t({\rm d}b)
  \\
  & \leq \iiint L \left| (A^t x+b) - (A^t x'+b) \right| P^0({\rm d}x) P^0({\rm d}x') P_b^t({\rm d}b)
  \\
  & = \iint L \|A\|^t \left| x-x' \right| P^0({\rm d}x) P^0({\rm d}x') 
  \\
  & \leq \int 2 L \|A\|^t  \left\| x \right\| P^0({\rm d}x)
  \\
  & = 2 L \|A\|^t  \mathbb{E}(\|X_0\|)
  \\
  & \leq \frac{ 2 L \|A\|^t  \mathbb{E}(\|X_0\|)}{1-\|A\|}.
 \end{align*}

Let us now check Assumption~\ref{asm:gamma:mixing}. In order to do so, fix $\ell\in\{1,\dots,n-1\} $ and a function $g:\mathcal{H}_k^\ell \rightarrow \mathbb{R} $ such that
\begin{equation}
\label{eq:hyp:g}
|g(a_1,\dots,a_\ell) - g(b_1,\dots,b_\ell)| \leq \sum_{i=1}^\ell \|a_i - b_i\|_{\mathcal{H}_k} .
\end{equation}
As we know that $\|\varepsilon_t\|\leq c$ a.s, we also have $\|X_t\|\leq c/(1-\|A\|)$. Fix $(x_1,\dots,x_\ell)$ with $\|x_t\|\leq c/(1-\|A\|)$. We have
\begin{align*}
 \Bigl| & \mathbb{E}[g(\mu_{\delta_{X_{\ell+1}}},\dots,\mu_{\delta_{X_{n}}})|X_{1}=x_1,\dots,X_{\ell}=x_\ell]-\mathbb{E}[g(\mu_{\delta_{X_{\ell+1}}},\dots,\mu_{\delta_{X_{n}}})]\Bigr|
 \\
 & =
 \Bigl| \mathbb{E}[g(\mu_{\delta_{A x_{\ell}+B_{\ell}^1}},\dots,\mu_{ \delta_{A^{n-\ell-1} x_{\ell}+B_\ell^{n-\ell-1} }})]-\mathbb{E}[g(\mu_{\delta_{A X_{\ell}+B_{\ell}^1}},\dots,\mu_{ \delta_{A^{n-\ell-1} X_{\ell}+B_{\ell}^{n-\ell-1}}})]\Bigr|
 \\
 & \leq
  \mathbb{E} \left[ \Bigl|g(\mu_{\delta_{A x_{\ell}+B_{\ell}^1}},\dots,\mu_{\delta_{A^{n-\ell-1} x_{\ell}+B_\ell^{n-\ell-1} }})-\mathbb{E}g(\mu_{\delta_{A X_{\ell}+B_{\ell}^1}},\dots,\mu_{\delta_{A^{n-\ell-1} X_{\ell}+B_{\ell}^{n-\ell-1}}})\Bigr|\right]
 \\
 & \leq \mathbb{E} \left[ \sum_{i=1}^{n-\ell-1} \left\| \mu_{\delta_{A^i x_{\ell}+B_{\ell}^i}} - \mu_{\delta_{A^{i} X_\ell + B_\ell^i}} \right\|_{\mathcal{H}_k} \right] \text{ by~\eqref{eq:hyp:g}}
 \\
 & =  \mathbb{E} \left[ \sum_{i=1}^{n-\ell-1} \sqrt{k(A^i x_\ell + B_{\ell}^i,A^i x_\ell + B_{\ell}^i) + k (A^i X_\ell + B_{\ell}^i,A^i X_\ell + B_{\ell}^i) - 2 k (A^i x_\ell + B_{\ell}^i,A^i X_\ell + B_{\ell}^i) } \right].
\end{align*}
We remind the assumption that $k(x,y)=F(|x-y|)$ where $F$ is $L$-Lipschitz. So, in particular:
$$ k(A^i x_\ell + B_{\ell}^i,A^i x_\ell + B_{\ell}^i) - k(A^i X_\ell + B_{\ell}^i,A^i x_\ell + B_{\ell}^i) \leq L \|A\|^i \|x_\ell - X_\ell\| $$
and in the same way, 
$$
 k(A^i X_\ell + B_{\ell}^i,A^i X_\ell + B_{\ell}^i) - k(A^i X_\ell + B_{\ell}^i,A^i x_\ell + B_{\ell}^i) \leq L \|a\|^i \|x_\ell - X_\ell\|, $$
so:
$$
 \Bigl|  \mathbb{E}[g(\mu_{\delta_{X_{\ell+1}}},\dots,\mu_{\delta_{X_{n}}})|X_{1}=x_1,\dots,X_{\ell}=x_\ell]-\mathbb{E}[g(\mu_{\delta_{X_{\ell+1}}},\dots,\mu_{\delta_{X_{n}}})]\Bigr|
   \leq  \mathbb{E} \left[ \sum_{i=1}^{n-\ell} \sqrt{2 L \|A\|^i \|x_\ell - X_\ell\| } \right].
$$
We now use $\|X_t\| \leq c/(1-\|A\|)$ a.s to get:
$$
 \Bigl|  \mathbb{E}[g(\mu_{\delta_{X_{\ell+1}}},\dots,\mu_{\delta_{X_{n}}})|X_{1}=x_1,\dots,X_{\ell}=x_\ell]-\mathbb{E}[g(\mu_{\delta_{X_{\ell+1}}},\dots,\mu_{\delta_{X_{n}}})]\Bigr|
   \leq   \sum_{i=1}^{n-\ell-1} \frac{2 L c \|A\|^{\frac{i}{2}}}{1-\|A\|} = \sum_{i=1}^{n-\ell-1} \gamma_i.
$$
This ends the proof.
\end{proof}

\subsection{Proofs of Section~\ref{sec:algo}}

\begin{proof}[Proof of Proposition~\ref{prop:gradient}]
 Note that we can rewrite
$$
{\rm Crit}(\theta) = \iint k(x,x') p_\theta(x) p_\theta(x') \mu({\rm d}x)\mu({\rm d}x') 
 - \frac{2}{n}\sum_{i=1}^n \int  k(x,X_i) p_\theta(x) \mu({\rm d}x).
$$
The assumption of the proposition ensure that we can interchange the $\nabla$ and $\iint$ symbols, and so
\begin{align*}
\nabla_\theta {\rm Crit}(\theta) & = \iint k(x,x') \nabla_\theta [p_\theta(x) p_\theta(x')] \mu({\rm d}x)\mu({\rm d}x') 
 - \frac{2}{n}\sum_{i=1}^n \int  k(x,X_i) \nabla_\theta [p_\theta(x)] \mu({\rm d}x)
 \\
 & = 2 \iint k(x,x') p_\theta(x) p_\theta(x') \nabla_\theta [\log p_\theta(x) ] \mu({\rm d}x)\mu({\rm d}x') - \frac{2}{n}\sum_{i=1}^n \int  k(x,X_i) \nabla_\theta [\log p_\theta(x)] p_\theta(x) \mu({\rm d}x)
 \\
 & = 2 \mathbb{E}_{X,X'\sim P_\theta} \{k(X,X') \nabla_\theta[\log p_\theta(X) ] \}
- \frac{2}{n}\sum_{i=1}^n \mathbb{E}_{X\sim P_\theta}\{k(X_i,X)\nabla_\theta[\log p_\theta(X) ]\}
\\
& = 2 \mathbb{E}_{X,X'\sim P_\theta} \left\{ \left[ k(X,X') - \frac{1}{n}\sum_{i=1}^n k(X_i,X) \right] \nabla_\theta[\log p_\theta(X) ]\right\}.
\end{align*}
This ends the proof.
\end{proof}

\begin{proof}[Proof of Proposition~\ref{prop:SGA}]
The assumption that $\Theta$ is bounded with radius $\mathcal{D}$ ensures that (2.17) in~\cite{Nemi2009} is satisfied, and the assumption on the expectation of the norm of the gradient ensures that (2.5) in~\cite{Nemi2009} is also satisfied. Thus,~(2.21) is also satisfied, ant that is exactly the statement of our~\eqref{eq:nemi}. Then, we have:
\begin{align*}
 \mathbb{E}\left[ \mathbb{D}_{k}\left(P_{\hat{\theta}^{(T)}_n},P^0\right) \right]
 & \leq  \mathbb{D}_{k}\left(P_{\hat{\theta}^{(T)}_n},\hat{P}_n \right) +  \mathbb{D}_{k}\left(\hat{P}_n,P^0 \right)
 \\
 & = \sqrt{\mathbb{D}_{k}^2\left(P_{\hat{\theta}^{(T)}_n},\hat{P}_n \right)} +  \mathbb{D}_{k}\left(\hat{P}_n,P^0 \right)
 \\
 & \leq \sqrt{\mathbb{D}_{k}^2\left(P_{\hat{\theta}_n},\hat{P}_n \right) + \frac{\mathcal{D}M}{\sqrt{T}}} +  \mathbb{D}_{k}\left(\hat{P}_n,P^0 \right)
\end{align*}
thanks to~\eqref{eq:nemi}. We upper bound the second term thanks to Lemma~\ref{lemma:mmd:1}:
$$ \mathbb{D}_{k}\left(\hat{P}_n,P^0 \right) \leq \sqrt{\frac{1+2\sum_{t=1}^n \varrho_t }{n}}. $$
For the first term, we use:
\begin{align*}
\sqrt{\mathbb{D}_{k}^2\left(P_{\hat{\theta}_n},\hat{P}_n \right) + \frac{\mathcal{D}M}{\sqrt{T}}}
  & \leq \mathbb{D}_{k}\left(P_{\hat{\theta}_n},\hat{P}_n \right) + \sqrt{\frac{\mathcal{D}M}{\sqrt{T}}} \\
  & \leq   \inf_{\theta\in\Theta} \mathbb{D}_{k}\left(P_{\theta},\hat{P}_n \right) + 2 \sqrt{\frac{1+2\sum_{t=1}^n \varrho_t }{n}}+ \sqrt{\frac{\mathcal{D}M}{\sqrt{T}}}
\end{align*}
thanks to Theorem~\ref{theorem:mmd:1}. Putting everything together leads to
$$
 \mathbb{E}\left[ \mathbb{D}_{k}\left(P_{\hat{\theta}^{(T)}_n},P^0\right) \right] \leq \inf_{\theta\in\Theta} \mathbb{D}_{k}\left(P_{\theta},\hat{P}_n \right) + 3 \sqrt{\frac{1+2\sum_{t=1}^n \varrho_t }{n}}+ \sqrt{\frac{\mathcal{D}M}{\sqrt{T}}}
$$
which ends the proof.
\end{proof}

\end{document}